\documentclass[10pt]{amsart}
\usepackage[active]{srcltx}
\usepackage[all
]{xy} \xyoption{poly}

\usepackage{xspace}
\usepackage{stmaryrd}
\usepackage{enumitem}
\usepackage[mathscr]{eucal}
\usepackage{amssymb}
\usepackage{a4wide}
\usepackage{ifthen}
\usepackage{amscd}
\usepackage[hyperindex,bookmarksnumbered,
plainpages,backref]{hyperref}

\newcommand{\g}{\mathfrak{g}}
\newcommand{\gh}{\hat{\mathfrak{g}}}
\newcommand{\gb}{\bar{\mathfrak{g}}}
\newcommand{\rr}{\mathfrak{r}}
\newcommand{\F}{\mathcal{F}}

\def\Z{\mathbb Z}

\newcommand{\kk}{{\bf k}}

\newcommand{\J}{{J}}

\newcommand{\cL}{{\mathcal{L}}}

\newcommand{\Jhalf}{\J\text{-mod}_{\frac 12}}
\newcommand{\JJ}{\J\text{-mod}_{1}}
\newcommand{\gm}{\hat\g\text{-mod}_{1}}
\newcommand{\ggm}{\g\text{-mod}_{1}}
\newcommand{\gmh}{\g\text{-mod}_{\frac 12}}

\newcommand{\so}{\mathfrak{so}}
\newcommand{\sll}{\mathfrak{sl}}

\newtheorem{thm}{Theorem}[section]
\newtheorem{prop}[thm]{Proposition}
\newtheorem{lem}[thm]{Lemma}
\newtheorem{rem}[thm]{Remark}
\newtheorem{Cor}[thm]{Corollary}

\newtheorem{ex}[thm]{Example}

\makeindex

\begin{document}

\title{On the Tits-Kantor-Koecher construction of unital Jordan bimodules} 

\author{Iryna Kashuba and Vera Serganova} 

\maketitle
\begin{abstract}
In this paper we explore relationship between representations of a Jordan algebra $\J$
and the Lie algebra $\g$ obtained from $\J$ by the Tits-Kantor-Koecher construction.
More precisely, we construct two adjoint functors $Lie :\JJ\to \ggm$ and $Jor:\ggm\to\JJ$, where $\JJ$ is the category of unital $\J$-bimodules and
$\ggm$ is the category of $\g$-modules admitting a short grading. Using these functors we classify $\J$ such that  its semisimple part is of Clifford type and 
the category $\JJ$ is tame.
\end{abstract}

\section{Introduction}

The famous Tits-Kantor-Koecher construction relates a unital Jordan (super)algebra $J$ with a Lie (super)algebra $\g$ equipped 
with a short $\mathbb Z$-grading. It was introduced independently in \cite{Tits}, \cite{Kantor} and \cite{Koecher} and one of most prominent  
applications was a classification of simple Jordan superalgebras in \cite{Kac1}, \cite{Kantor2}, \cite{Kac-Zelmanov}.

The TKK construction has been proven to be quite efficient and useful in the study of Jordan superalgebras, Jordan superpairs and their superbimodules.
Several application of TKK construction in representation theory of semisimple Jordan superalgebras and Jordan superpairs can be found in 
\cite{Shtern}, \cite{Shtern2}, \cite{Mart-Zel}, \cite{Krut} and \cite{Krut2}. 
The goal of this paper is to further study and apply this construction to non-semisimple 
Jordan algebras and their representations. 

%Before we continue to relate  we discuss in more details a representation theory for Jordan algebras. 
%It has two closely related aspects both of them are concerned with mappings of 
%Jordan algebras into associative algebras. The first deals with  homomorphism of Jordan algebras into special ones.  
%This is a linear mapping $\sigma: J\to A$, such that $\sigma(a\circ b)=\frac12(\sigma(a)\sigma(b)+\sigma(b)\sigma(a))$, $A$ an associative algebra with $1$
%and $a\circ b$ multiplication in $J$. 

Recall that a representation of a Jordan algebra $\J$ in a vector space $M$ is a linear mapping $\rho$ of $\J$ into $\operatorname{End}_{\kk}(M)$ such that 
\begin{equation}\label{relations-unital-module}
[\rho(a),\rho(a^2)]=0, \quad 2\rho(a)\rho(b)\rho(a)+\rho((a\circ a)\circ b)
=2\rho(a)\rho(a\circ b)+\rho(b)\rho(a\circ a).
\end{equation}
The category of finite dimensional $\J$-modules will be denoted by $\J$-mod.

%It is remarkable that the notion of associative specialization and multiplication specialization which look 
%different formally are in fact closely related. The relation between two notions is the folowing. Suppose $\sigma_1$
%and $\sigma_2$ are two associative specializations in $A$ and assume that $[\sigma_1(a),\sigma_2(b)]=0$, $a,b\in J$.
%Then the average $\rho=\frac12(\sigma_1+\sigma_2)$ is a multiplication speciaization of $\J$ in $A$. 

%In particular, every
%associative specialization $\,\sigma:\J\to {\operatorname End}_{\kk} M$ endows
%$\,M\,$ with a structure of $\,\J$-module via
%$\,\frac12\phi:\cJ\to{\operatorname End}_{\kk}M$. The module of this type we
%will call \emph{one-sided}(?????). Let $\J$ be a
%unital Jordan algebra with the unit $e$. A $\,\J$-module $\,M\,$
%is called {\it unital} if $\,e\cdot m=m\,$ for all $\,m\in M$.
%The full subcategories of unital $\J$-bimodules
%and of one-sided $\J$-bimodules will be denoted
%$\JJ$ and $\J$-mod$_{\frac12}$ correspondingly. 

If $\J$ is a unital algebra the category of $J$-modules has a decomposition into the direct sum of three subcategories
$$J\text{-mod}=\JJ\oplus \Jhalf\oplus J\text{-mod}_0$$
according to the action of the identity element of $J$. The subcategory $ J\text{-mod}_0$ is not interesting since all its 
objects are trivial modules.
The subcategory $\Jhalf$ consists of modules on which the identity element acts as $\frac12$, such modules are called special.
The objects of $\JJ$ are called unital modules. One can introduce associative enveloping algebras for $\J$-mod, $\JJ$ and $\Jhalf$,
such that each of these categories is equivalent to the category  of modules over  the corresponding enveloping algebra.

Recall the classification of simple Jordan algebras over an algebraically closed field $\kk$ of characteristic zero. 
With the exception of the case
$\J=\kk$, simple Jordan algebras are divided in two groups: Jordan algebras of quadratic form $\J(E,q)$, see Section 5 
for details, and Jordan algebras of matrix type, see \cite{Jac}. The latter are called sometimes Hermitian Jordan algebras.
 
In \cite{Jac2} Jacobson constructed the associative enveloping algebras for $\JJ$ and $\Jhalf$, when $\J$ is finite-dimensional 
simple,
and proved that both categories are semisimple with finitely many simple objects.

The next step is to study non-semisimple Jordan algebras. In this case it is important to classify tame 
categories $\Jhalf$ and $\JJ$  
(for basics on tame and wild categories see \cite{drozd}).
In \cite{KOSh} the enveloping algebra of $\Jhalf$ was studied in the case when the semisimple part of $\J$ is of matrix type 
and $rad^2\J=0$.
Using the coordinalization theorem for Jordan algebras of matrix type the authors  proved that the enveloping 
algebra and 
consequently the 
Ext quiver algebra of $\Jhalf$ have radical squared equal to zero. Hence they could employ the representation theory of 
quivers to classify tame $\Jhalf$. 

In all other cases the above method is not applicable. But it seems likely that we can later deal with the remaining cases using the TKK 
construction. 
The main advantage of this approach is the existence of a tensor structure on the category of $\g$-modules and a well 
developed theory of weights. 

In this paper we focus on Jordan algebras,
whose semisimple part is a sum of Jordan algebras of quadratic forms.
We classify all such algebras with tame $\JJ$ without any additional assumptions on the radical, see Theorem \ref{thm-general-case}. 
For this purpose we avoid cases of small dimensions: we start with simple Jordan algebras of dimension greater than $4$.
It follows from our classification that all such tame Jordan algebras $\J$ satisfy the condition $rad^2J=0$.
On the other hand, in contrast with  \cite{KOSh},  the square of the radical of the universal enveloping algebra is not 
necessarily zero for tame categories.
The category $\Jhalf$ is studied in a forthcoming paper \cite{KS}.

In Section 3 we define and study two adjoint functors $Jor$ and $Lie$ between the category $\JJ$ and the category 
$\ggm$ of $\g$-modules admitting a short grading.
The  definition of $Jor$ is straightforward. However, not every $J$-module can be obtained from a $\g$-module by application of $Jor$. 
To fix this flaw one has to consider 
the universal central extension $\gh$ 
of $\g$. This problem does not appear in the semisimple case since $\gh=\g$ but it is already essential for simple Jordan and 
Lie superalgebras, see \cite{Mart-Zel}.
Although algebras with non-zero central extensions do not appear in our classification, we formulate statements in full generality 
for future applications.
The second problem worth mentioning here is caused by the fact that the splitting $\JJ\oplus J\text{-mod}_0$ can not be 
lifted to the Lie algebra $\gh$, since 
some modules can have non-trivial extensions with trivial modules. That implies, in particular, that left and right adjoint of the 
functor $Jor$ are not isomorphic 
and the categories $\gm$ and $\JJ$ are not equivalent. Still they are close enough and one can describe projective modules, 
quivers and relations of $\JJ$ in terms of $\gm$. 

In Section 4 we explain how to construct the Ext quivers of $\gm$ and $\JJ$ and compute the radical filtration of projective 
indecomposable modules.
 
In Sections 5--9 we classify Jordan algebras with tame categories of unital representations satisfying above mentioned conditions. 
Our main tool is the representation theory of quivers.
All the quiver results we use are collected in Appendix. 
Although our algebras do not satisfy the condition $rad^2=0$, we use a lot Theorem \ref{Viktor} 
which could be considered as a generalization of this property.
Finally, let us mention that all tame associative algebras $A$ arising from our classification are quadratic and satisfy 
the conditions 
$rad^3A=0$ and $A\simeq A^{op}$.

\section{Tits-Kantor-Koecher construction for Jordan algebras.}
\subsection {Jordan algebras and bimodules} Let $\kk$ be a field, char $\kk\neq 2$.
A Jordan {\bf k}-algebra is a commutative algebra $\J$ such that any $\,a,b\in
\J$ satisfy the Jordan identity:
\begin{eqnarray}\label{definition-jordan-algebra}
 a\circ b &=&b \circ a\\
((a\circ a)\circ b)\circ a&=&(a\circ a)\circ (b\circ a).
\end{eqnarray}
For any associative algebra $\,A\,$ one can construct the Jordan
algebra $\,A^+\,$ by introducing on a vector space $\,A\,$ a new
multiplication $a_{1}\circ a_{2}=\frac12(a_{1}a_{2}+a_{2}a_{1})$.
If a Jordan algebra is isomorphic to a subalgebra of the algebra
$A^+$ for a certain associative algebra $A$ then it is called {\em
special}, otherwise it is {\em exceptional}.

Let $\J$ be a Jordan algebra over ${\bf k}$ and $M$  be 
a $\kk$-vector space endowed with a pair of linear mappings
$l:\J\otimes_{{\bf k}}M\to M $, $\,(a\otimes m)\mapsto a\cdot m$,
$r:M\otimes_{\bf k}\J\to M $, $\,(m\otimes a)\mapsto m\cdot a$,
$\,a\in \J$, $m\in M$. Then $M$ is called a {\it Jordan bimodule} over $\J$ if
the algebra
$\,Z=(\J\oplus M, \ast)$,  where  $\ast$ is a ${\bf k}-$bilinear product
$$ (a_1+m_1)\ast(a_2+m_2)=a_1\circ
a_2+a_1\cdot  m_2+m_1\cdot a_2,
$$
for $\,a_1,a_2\in \J,\,$ $\,m_1,m_2\in M$, is a Jordan  algebra. Observe that
$\J$ is a subalgebra in $Z$ and $M$ is an
ideal with $\,M^2=0$.  In this case
$Z$ is called the {\em null split extension} of $\,\J\,$
by the bimodule $M$. 
It follows from the Jordan identity \eqref{definition-jordan-algebra} that if $M$ is a Jordan bimodule over $\J$
the corresponding representation $\rho: \J \to\operatorname{End}_{\kk}M$
satisfies \eqref{relations-unital-module}.

%If $M$ is a Jordan bimodule over $\J$, then one can consider the
%mapping $\rho:\J\to {\rm End}_{\bf k} M$, $\rho(a)(m)=a\cdot m,\
%a\in \J,\ m\in M$. Due to \eqref{equation-Jordan-bimodule2},
%\eqref{equation-Jordan-bimodule1}, $\rho$ satisfies the identities
%\begin{equation}\label{equation-Jordan-representation}
%[\rho(a),\rho(a\circ a)]=0,\quad
%2\rho(a)\rho(b)\rho(a)+\rho((a\circ a)\circ b)
%=2\rho(a)\rho(a\circ b)+\rho(b)\rho(a\circ a).
%\end{equation}
%Any mapping  $\rho:\J\to A$  into an associative algebra $A$,
%satisfying  relations \eqref{equation-Jordan-representation} is
%called a {\em multiplicative specialization} of $\J$. Evidently,
%any multiplicative specialization $\rho:\J\to {\rm End}_{\bf k} M$
%endows $M$ with a structure of a $\J$-module.

\subsection{TKK construction}
A {\em short grading} of an algebra $\g$ is a $\Z$-grading of the
form $\g=\g_{-1}\oplus\g_0\oplus\g_1$. Let $P$ be the commutative
bilinear map on $\J$ defined by $\,P(x,y)=x\circ y$. Then we
associate to $\J$ a Lie algebra $\g=Lie(\J)$ with short grading
$\g=\g_{-1}\oplus\g_0\oplus\g_1$ in the following way, see \cite{Kac}. We 
set $$\g_{-1}=\J, \quad\g_0=\langle L_a,[L_a,L_b]\,|\,a,b\in\J\rangle\subset \operatorname{End}_\kk(J),$$ 
where $L_a$ denotes the operator of left multiplication in $\J$,
and $$\g_1=\langle P,[L_a,P]\,|\,a\in\J\rangle\subset\operatorname{Hom}_\kk(S^2J,J)$$ with the following
bracket
\begin{itemize}
\item $[x,y]=0$ for $x,\,y\in \g_{-1}$ or $x,\,y\in \g_1$; \item
$[L,x]=L(x)$ for $x\in\g_{-1}$, $L\in\g_0$; \item
$[B,x](y)=B(x,y)$ for $B\in\g_1$ and $x,y\in\g_{-1}$; \item
$[L,B](x,y)=L(B(x,y))-B(L(x),y)+B(x,L(y))$ for any $B\in\g_1$,
$L\in\g_0$ and $x,y\in\g_{-1}.$
\end{itemize}
$Lie(\J)$ is a Lie algebra.

Note that by construction $Lie(\J)$ is generated as a Lie algebra
by $Lie(\J)_1\oplus Lie(\J)_{-1}$.

A {\em short subalgebra} of a Lie algebra $\g$ is an $\sll_2$
subalgebra spanned by elements $e,h,f$, satisfying $[e,f]=h, [h,e]=-e, [h,f]=f$, such that the eigenspace
decomposition of $ad\,h$ defines a short grading on $\g$. Consider a
Jordan algebra $\J$ with a unit element $e$. Then  $e$, $h_{\J}=-L_e$ and $f_{\J}=P$ span
a short subalgebra $\alpha_J\subset Lie(\J)$. A
$\Z$-graded Lie algebra $\g=\g_{-1}\oplus\g_0\oplus\g_1$ is called
{\em minimal} if any non-trivial ideal $I$ of $\g$ intersects
$\g_{-1}$ non-trivially, i.e. $I\cap \g_{-1}$ is neither $0$ nor
$\g_{-1}$.

\begin{lem}\cite{Kac}
A Lie algebra $\g=\g_{-1}\oplus\g_0\oplus\g_1$ is minimal if and
only if the following conditions hold: \begin{enumerate}\item if
$[a,\g_{-1}]=0$ for some $a\in\g_0\oplus\g_1$, then $a=0$; \item
$[\g_{i},\g_0]=\g_{i},  i=\pm 1$.\end{enumerate}
\end{lem}

 Let $\mathcal{J}$ denote the category of unital Jordan algebras in
which morphisms are Jordan epimorphisms and let $\cL$ denote the
category of minimal pairs $(\g,\alpha)$, where $\g$ is a Lie
algebra, $\alpha$ a short subalgebra of $\g$, and a morphism
$\phi$ from pair $(\g,\alpha)$ to $(\g',\alpha')$ is a Lie algebra
epimorphism $\phi:\g\to\g'$ such that $\phi(\alpha)=\alpha'$. We
construct a functor $\F:\mathcal{J}\to\cL$ by associating to
every unital Jordan algebra $\J$ the pair $(Lie(\J),\alpha_{\J})$
and to every epimorphism $\phi:\J\to\J'$ of unital Jordan algebras
the map $\phi_{\F}: Lie(\J)\to Lie(\J')$ defined as follows:
\begin{equation}\label{multiplication_in_Lie(J)}
x\mapsto \phi(x),\qquad L_{x}\mapsto L_{\phi(x)}, \ {\rm for\ }
x\in\J; \qquad P\to P',\quad P'(x,y)=x\circ_{\J'} y.
\end{equation}
Let $\g$ be a Lie
algebra containing an $\sll_2$-subalgebra $\alpha$ which induces a short grading
$\g=\g_{-1}\oplus\g_0\oplus\g_1$. Then $\J=(\g_{-1}, \circ_{\J})$ with
$x\circ_{\J} y=[[P,x],y]$ is a Jordan algebra. Moreover, any
epimorphism $\phi: (\g,\alpha)\to(\g',\alpha')$ in $\F$ defines
Jordan algebra epimorphism  $\phi_{|\g_{-1}}$. Thus, we have defined
a functor $\cL\to{\mathcal J}$ which we denote by $Jor$.
The functors $\F$ and $Jor$ define  an equivalence of categories  $\cL$ and ${\mathcal J}$, see Theorem 5.15, \cite{Kac}.

\section{Functors $Lie$ and $Jor$ for unital modules}

Let $\J$ be a unital Jordan algebra and $\g=Lie(\J)$. By $\hat{\g}$ we denote the universal central extension of $\g$.
Note that $\hat{\g}$ contains the $\sll_2$-subalgebra $\alpha=\langle e,h,f\rangle$, hence the center of $\hat{\g}$ is in $\hat{\g}_0$.
It implies that
$$\g_{\pm 1}=\hat\g_{\pm 1}.$$
Let $\gm$ denote the category of $\hat{\g}$-modules $N$  such that
the action of $h\in\alpha$ induces a grading of length $3$ on $N$.
We will construct two functors
$$Jor:\gm\to\JJ,\quad Lie:\JJ\to\gm$$
and show that $Lie$ is left adjoint to $Jor$.

To define $Jor$ let $N\in\gm$. Then $N$ has a short grading $N=N_{1}\oplus N_{0}\oplus N_{-1}$.
We set $Jor(N):=N_{-1}$ with action of $\J=\g_{-1}=\hat\g_{-1}$ defined by
$$x(m)=[f,x]m,\quad x\in\J=\g_{-1}, m\in N_{-1}.$$ 
It is clear that $Jor$ is an exact functor.

Our next step is to define $Lie:\JJ\to\gm$. Let $M\in\JJ$. Consider the associated null split extension
$\J\oplus M$. Let $\mathcal A=Lie (\J\oplus M)$. Then by \cite{Kac} we have an exact sequence of Lie algebras
\begin{equation}\label{exact1}
0\to N\to\mathcal A\xrightarrow{\pi} \g\to 0,
\end{equation}
where $N$ is an abelian Lie algebra and $N_{-1}=M$.

\begin{lem}\label{centext} Let $\gamma: \hat\g\to \g$ be the canonical projection. There exists $s:\hat\g\to\mathcal A$
such that $\pi\circ s=\gamma$.
\end{lem}
\begin{proof} Observe that the splitting $\mathcal A_{\pm 1}=\g_{\pm 1}\oplus N_{\pm 1}$ is canonical.
Let $\tilde\g$ be the Lie subalgebra in $\mathcal A$ generated by $\g_{\pm 1}$. 
Then we have a surjective homomorphism $\varphi:\tilde\g\to\g$ and $\operatorname{Ker}\varphi\subset \tilde\g_0$.
We claim that $\operatorname{Ker}\varphi$ lies in the center of $\tilde\g$.
Indeed, $z\in\operatorname{Ker}\varphi$ implies $[z,\tilde{\g}_{\pm 1}]\subset \tilde{\g}_{\pm 1}\cap\operatorname{Ker}\varphi=0$ and from
$[\tilde{\g}_{-1},\tilde{\g}_1]=\tilde{\g}_0$ it follows that
$[z,\tilde{\g}_0]=0.$
Therefore the map $s:\hat\g\to\tilde\g\subset\mathcal A$ is as required.
\end{proof}

\begin{rem} For the illustration that $\gh$ is essential, see Example \ref{centexs}. 
\end{rem}

The above Lemma implies that  $N$ is a $\hat\g$-module. Thus, in particular, we have defined a $\hat\g_0$-module structure on $N_{-1}=M$.
Now let $\mathcal P=\hat\g_0\oplus\g_{-1}$ and we extend the above  $\hat\g_0$-module structure on $M$ to a $\mathcal P$-module structure
by setting $\g_{-1} M=0$. Let
$$\Gamma (M)=U(\hat\g)\otimes_{U(\mathcal P)}M.$$
We define $Lie(M)$ to be the maximal quotient in $\Gamma(M)$ which belongs to $\gm$. More precisely
$Lie(M):=\Gamma(M)/T$, where $T$ is the submodule in $\Gamma(M)$ generated $\displaystyle \bigoplus_{i\geq 2}\Gamma(M)_i$.

Note that Frobenius reciprocity implies that for any $K\in\gm$ and any $M\in\JJ$ 
\begin{equation}\label{frobenius}
\operatorname{Hom}_{\hat\g}(Lie(M),K)\simeq\operatorname{Hom}_{\hat\g}(\Gamma(M),K)\simeq \operatorname{Hom}_{\mathcal P}(M,K).
\end{equation}
On the other hand, we have
\begin{equation}\label{frobenius2} 
\operatorname{Hom}_{\mathcal P}(M,K)=\operatorname{Hom}_{\gh_0}(M,K_{-1})= \operatorname{Hom}_{\J}(M,Jor(K)).
\end{equation}
\begin{lem}\label{adjoint} We have a canonical isomorphism
$$\operatorname{Hom}_{\hat\g}(Lie(M),K)\simeq \operatorname{Hom}_{\J}(M,Jor(K)),$$
\end{lem}
\begin{proof} Indeed,
$$\operatorname{Hom}_{\hat\g}(Lie(M),K)\simeq \operatorname{Hom}_{\mathcal P}(M,K)\simeq \operatorname{Hom}_{\J}(M,Jor(K)).$$
where the first isomorphism is a consequence of (\ref{frobenius}) and the second follows from (\ref{frobenius2}).
\end{proof}

\begin{Cor} If $P$ is a projective module in $\JJ$, then $Lie(P)$ is a projective module in $\gm$.
\end{Cor}
\begin{proof} Follows from Lemma \ref{adjoint} and exacteness of $Jor$.
\end{proof}

\begin{lem}\label{identity} $Jor\circ Lie$ is isomorphic to the identity functor in $\JJ$.
\end{lem}
\begin{proof} By construction we have $Jor\circ Lie(M)=(Lie(M))_{-1}\simeq M$.
\end{proof}

\begin{lem}\label{induction}  Let $N\in\gm$. We have an exact sequence of $\hat\g$-modules
\begin{equation}\label{exact2}
0\to C\to Lie(Jor(N))\to N\to C'\to 0,
\end{equation}
where $C,\,C'$ are some trivial $\hat \g$-modules.
\end{lem}
\begin{proof} The identity morphism $Jor(N)\to Jor (N)$ 
induces a homomorphism of $\gm$-modules $ Lie(Jor(N))\to N$ by Lemma \ref{adjoint}. Let $C$ and $C'$ denote the kernel and cokernel of this homomorphism.
Then we obtain the sequence (\ref{exact2}). Apply $Jor$ to this sequence. Since $Jor(Lie(Jor(N)))\simeq  Jor(N)$, exactenes of $Jor$ implies $Jor(C)=Jor(C')=0$. 
Therefore
$C$ and $C'$ are trivial $\hat\g$-modules.
\end{proof}

\begin{Cor}\label{aux1}
(a) Let $N\in \gm$ and $\hat\g N=N$, then the canonical map $Lie(Jor(N))\to N$ is surjective.

(b) Let $N\in \gm$ and $N^{\hat\g}:=\{x\in N\,|\, \hat\g x=x\}=0$, then the canonical map $N\to Lie(Jor(N))$ is injective.

(c) If $M\to L\to 0$ is exact in $\JJ$, then $Lie(M)\to Lie(L)\to 0$ is exact in $\gm$;
\end{Cor} 
\begin{proof} Note that (a) and (b) follow from Lemma \ref{induction} since in (a) we have $C'=0$ and in (b) we have $C=0$.
To prove (c) consider the exact sequence $Lie(M)\to Lie(L)\to C\to 0$, where $C$ is the cokernel of $Lie(M)\to Lie(L)$ and apply $Jor$.
Then again we have $Jor(C)=0$. Note that by construction $Lie(L)$ is generated by $L=Lie(L)_{-1}$ and therefore $Lie(L)$ does not have a trivial quotient.
Hence $C=0$.
\end{proof}

\begin{lem}\label{proj} Let $P$ be a projective module in $\gm$
such that $\hat{\g}P=P$. Then $Jor(P)$ is projective in $\JJ$.
\end{lem}
\begin{proof} By Corollary \ref{aux1} (a) we have a surjection $Lie(Jor(P))\to P$. However, $P$ is projective and that gives an isomorphism
$Lie(Jor(P))\simeq P$. Now let $M\to N\to 0$ be an exact sequence in $\JJ$. We rewrite it in the form $Jor(Lie(M))\to Jor(Lie(N))\to 0$.
Now we use 
$$\operatorname{Hom}_{\J}(Jor(P),Jor(Lie(M)))\simeq \operatorname{Hom}_{\hat\g}(Lie(Jor(P)),Lie(M)),$$
$$\operatorname{Hom}_{\J}(Jor(P),Jor(Lie(N)))\simeq \operatorname{Hom_{\hat\g}}(Lie(Jor(P)),Lie(N)).$$
By Lemma \ref{aux1}(c) $Lie(M)\to Lie(N)\to 0$ is exact, hence we have that 
$$\operatorname{Hom}_{\hat\g}(P,Lie(M))\to \operatorname{Hom}_{\hat\g}(P,Lie(N))\to 0$$
is also exact. Application of $Jor$ implies exactness of 
$$\operatorname{Hom}_{\J}(Jor(P),M)\to \operatorname{Hom}_{\J}(Jor(P),N)\to 0.$$
The latter is projectivity of $Jor(P)$.
\end{proof}

\begin{Cor}\label{Jorsimple} Let $L\in\gm$ be simple and non-trivial. Then $Jor(L)$ is simple. 
\end{Cor}

\begin{Cor}\label{Jorradical} Let $M\in\gm$. If $(M/rad M)^{\gh}=0$, then $Jor (rad M)=rad Jor(M)$.
\end{Cor}

\begin{rem} One can construct functors $Lie$ and $Jor$ between the category of special Jordan modules and the category of 
$\gh$-modules
with grading of length $2$. In this case these functors establish an equivalence of categories, see \cite{KS} for details.
\end{rem}

\section {On the categories $\ggm$ and $\gmh$}

In the rest of the paper we assume that the ground field $\kk$ is algebraically closed of characteristic zero.
Let $\g$ be a finite-dimensional Lie algebra which contains an $\sll_2$-subalgebra $\alpha=\langle e,h,f\rangle$ with short 
grading $\g=\g_{-1}\oplus\g_0\oplus\g_1$ induced by the action of $h$.
We fix a  Levi subalgebra $\g_{ss}\subset \g$ such that $\alpha\subset \g_{ss}$ and denote by $\rr$ the 
radical of $\g$. 
Then $\g$ is a semi-direct sum $\g_{ss}\niplus\rr$.
We assume in addition that $$\rr^{\g_{ss}}\subset [\g,\g]\cap Z(\g)$$ and $\g$ is generated by $\g_1$ and $\g_{-1}$. 
These assumptions imply that $\rr$ is a nilpotent Lie algebra. Define a decreasing filtration 
$$\rr=\rr^1\supset\rr^2\supset\dots$$
by setting $\rr^{i}:=[\rr,\rr^{i-1}]$ for all $i>1$.
Let $R^i=\rr^i/\rr^{i+1}$ and write $R=R^1=\rr/\rr^2$ to simplify notation.

Let $\mathcal S$ be the full subcategory of finite-dimensional $\g$-modules consisting of all modules  $M$ such that 
$$M=M_{-1}\oplus M_{-\frac{1}{2}}\oplus M_{0}\oplus M_{\frac{1}{2}}\oplus M_1$$
in the grading induced by the action of $h$.
In this section we prove some general statements about $\mathcal S$.

We notice first that $\mathcal S=\ggm\oplus\gmh$ is a direct sum of categories, and all simple objects in $\mathcal S$ are simple as $\g_{ss}$-modules.
In what follows we denote by $tr$ the trivial simple $\g$-module.

\begin{rem}\label{duality1} It is useful to notice that the category $\mathcal S$ has a contravariant duality functor which sends $M$ to $M^*$.
In particular, $\mathcal S$ is equivalent to $\mathcal S^{op}$.
\end{rem}

\begin{rem}\label{tensor} We also note that the tensor product of two modules from $\gmh$ is a module from $\ggm$.
So we have a bifunctor $\gmh\times\gmh\to \ggm$. In the language of Jordan algebras that corresponds to the well-known bifunctor
$\Jhalf\times\Jhalf\to \JJ$, see for example \cite{Jac}, Section II.10.
\end{rem}

\subsection{Indecomposable projectives and Ext quivers}\label{subsection-plus} 
Consider the category $(\g,\g_{ss})\text{-mod}$ of all $\g$-modules $M$ integrable over $\g_{ss}$. 
Note that $\g_{ss}$-integrability implies that the action of $h$ is semisimple and eigenvalues of $h$ are in 
$\frac{1}{2}\mathbb Z$. 
Clearly,  $\mathcal S$ is a full abelian subcategory in  $(\g,\g_{ss})\text{-mod}$.  
We define a functor $${}^{sh}:(\g,\g_{ss})\text{-mod}\to \mathcal S$$ by setting $M^{sh}=M/N$ where $N$ is the submodule generated by all 
graded components in $M$ of degree 
greater or equal than $\frac{3}{2}$. Obviously, we have
$$\operatorname{Hom}_{\g}(M^{sh},K)=\operatorname{Hom}_{\g}(M,K)$$
for any $M\in (\g,\g_{ss})\text{-mod}$ and $K\in\mathcal S$. In other words, ${}^{sh}$ is left adjoint to the embedding functor  
$\mathcal S\to (\g,\g_{ss})\text{-mod}$.
That implies in particular, that if $P$ is projective in $(\g,\g_{ss})\text{-mod}$, then $P^{sh}$ is projective in $\mathcal S$. 

To construct the projective cover $P(L)$ of a simple module
$L\in \mathcal S$ consider the induced module 
$$I(L)=U(\g)\otimes _{U(\g_{ss})}L.$$
By Frobenius reciprocity $I(L)$ is the projective cover of $L$ in  $(\g,\g_{ss})\text{-mod}$. Therefore
$P(L)=I(L)^{sh}$ is the projective cover of $L$ in $\mathcal S$.

\begin{lem}\label{findim} $P(L)$ is finite-dimensional.
\end{lem}
\begin{proof} Let $M=P(L)$. Recall that $S(\rr)$ is the associated graded algebra of the universal enveloping algebra $U(\rr)$ with 
respect to the PBW filtration. 
Let $Gr M$ be the corresponding graded $S(\rr)$-module. Note that $Gr M$ inherits the short grading of $M$ and is generated by $L$.
Let $I=\operatorname{Ann}_{S(\rr)} L$.
Then $I$ is a $\g_{ss}$-invariant ideal, and we have $Gr M\simeq (S(\rr)/I)\otimes L$. Consider 
the $\g_{ss}$-invariant decomposition $\rr=\rr'\oplus Z(\g)$. Let $I'=I\cap S(\rr')$. It is well know fact of representation theory of semisimple algebraic groups,
that if $\operatorname{dim}S(\rr')/I'=\infty$ then there exists a $\g_{ss}$-highest vector $v\in\rr'$ such that $v^m\notin I'$ for all $m>0$. 
That excludes the possibility
$\operatorname{dim}S(\rr')/I'=\infty$ since $v\in \rr'_1$ and the induced by action of $h$ grading of $S(\rr')/I'$ must be bounded. 

Thus, we have $\operatorname{dim}S(\rr')/I'<\infty$, which implies
$I'\supset (\rr')^k$ for some $k>0$. Therefore we have
$$M=U(Z(\g))\sum_{i=0}^k (\rr')^i L,$$
where $Z(\g)$ denotes the center of $\g$.

\noindent Since by our assumptions on $\g$ we have $Z(\g)\subset [\rr',\rr']$ we obtain that for sufficiently large $n$
$$M=\sum_{i=0}^n \rr^i L.$$
Therefore $M$ is finite-dimensional. 
\end{proof}

Let $\gb:=\g/[\rr,\rr]$. Then $\gb=\g_{ss}\niplus R$ where $R=\rr/[\rr,\rr]$ is the abelian radical of $\g$.
We denote by $Q(\mathcal S)$ the Ext quiver of the category $\mathcal S$, by $Q(\g)$ the Ext quiver of $\ggm$ and by $Q^{\frac12}(\g)$ the Ext quiver
of  $\gmh$. Clearly, $Q(\mathcal S)$ is the disjoint union of $Q(\g)$ and  $Q^{\frac12}(\g)$.  

\begin{lem}\label{quiver_Lie} Let $L$ and $L'$ be two simple modules in $\mathcal S$.
Then 
$$\operatorname{Ext}_{\g}^1(L',L)=\operatorname{Ext}_{\gb}^1(L',L)$$
and $\operatorname{dim}\operatorname{Ext}_{\g}^1(L',L)$ equals the
multiplicity of $L$ in $L'\otimes R$.
\end{lem}
\begin{proof} Consider a non-trivial extension of $L$, $L'\in\mathcal S$
$$0\to L\to M\to L'\to 0,$$ then $\rr M\subset L$ and $\rr L=0$. Therefore $[\rr,\rr]M=0$, and hence $M$ is a $\gb$-module.
That implies the first assertion.
 
Now we use the fact that $M$ splits over $\g_{ss}$, and the action $R\otimes L' \to L$ is a $\g_{ss}$-invariant map. 
Therefore $\operatorname{Ext}^1_{\g}(L',L)\simeq\operatorname{Hom}_{\g_{ss}}(L'\otimes R,L)$.
\end{proof}
\begin{Cor}\label{hat=bar} $Q(\g)=Q(\gb)$ and $Q^{\frac12}(\g)=Q^{\frac12}(\gb)$.
\end{Cor}

Let $P$ (respectively, $P'$) be the direct sum of $P(L)$ over all up to isomorphism simple $L$ in $\ggm$ 
(respectively, $\gmh$),
$A(\g):=\operatorname{End}_{\g}(P)$, $A^{\frac12}(\g):=\operatorname{End}_{\g}(P')$.

By general results usually attributed to Gabriel (see Appendix) we know that $\ggm$ and $\gmh$ are equivalent to the categories of 
finite-dimensional right $A(\g)$-modules and $A^{\frac12}(\g)$-modules respectively.

\subsection{Radical filtration of indecomposable projectives} In what follows we will need a description of the first three layers of the radical
filtration of an indecomposable projective $P(L)$. To simplify notations we set $P^k(L):=rad^k P(L)/rad^{k+1} P(L)$.
We have $L=P(L)/rad P(L)$ by definition and $P^1(L)=(R\otimes L)^{sh}$ by Lemma\,\ref{quiver_Lie}.

Let $U(\rr)$ be the universal enveloping algebra of $\rr$ and $\mathfrak{R}=\langle\rr\rangle$ denote the augmentation ideal.
We observe first that $I(L)\simeq U(\rr)\otimes L$ as a module over $\rr$.  
Since the action of $\rr$ is nilpotent on all modules in the category $(\g,\g_{ss})$-mod, we obtain that
$$rad^kI(L)=\mathfrak R^k\otimes L.$$
Since $P(L)=I(L)^{sh}$ is a quotient of $I(L)$ we obtain 
$$rad^kP(L)=\mathfrak R^k P(L).$$

We proceed to describing  $P^2(L)$.
Let $L$ be a simple $\g_{ss}$-module and $$\pi:R\otimes R\otimes L\to (R\otimes (R\otimes L)^{sh})^{sh}, \quad p:R^2\otimes L\to(R^2\otimes L)^{sh}$$ 
be the maps induced by the canonical projection $X\to X^{sh}$. 
Consider also the maps $$\delta: \Lambda^2 R\to R^2,\quad\delta(x,y):=-[x,y]\,\mod\rr^3$$ and 
$$alt:\Lambda^2 R \to R\otimes R,\quad alt(x,y):=x\otimes y-y\otimes x.$$

\begin{lem}\label{technical2} Let $L$ be a semisimple $\g_{ss}$-module. Consider the maps
\begin{equation}\label{mu-pi-relation}
\mu: \Lambda^2R\otimes L \xrightarrow{alt\otimes 1}R\otimes R\otimes L\xrightarrow{\pi}(R\otimes (R\otimes L)^{sh})^{sh}
\end{equation}
and
$$\lambda: \Lambda^2 R\otimes L\xrightarrow{\delta\otimes 1} R^2\otimes L\xrightarrow{p} (R^2\otimes L)^{sh}.$$
Then $P^2(L)$ is isomorphic to the cokernel of $\mu\oplus\lambda$.
\end{lem}
\begin{proof}  The universal enveloping algebra $U(\rr)$ is the quotient of the tensor algebra $T(\rr)$ by the ideal
generated by $x\otimes y-y\otimes x-[x,y]$. In particular, at the second layer of the augmentation  filtration we have
$${\mathfrak R}^2/{\mathfrak R}^3\simeq\left(R\otimes R\oplus R^2\right)/\left(alt(x\otimes y)+\delta(x,y)\right)_{x,y\in R}.$$
Therefore $I^2(L):=rad^2 I(L)/rad^3 I(L)$ is the cokernel of $(alt\otimes 1)\oplus(\delta\otimes 1)$.

Thus, the statement follows from the commutative diagram
\begin{equation}\xymatrix{{R\otimes R\otimes L}\ar[r]\ar[d]&{(R\otimes(R\otimes L)^{sh})^{sh}}\ar[d]\\
I^2(L)\ar[r]&P^2(L)}
\end{equation}
\end{proof}

\begin{ex}\label{centexs} Let $\g=\gb=\sll_2\niplus R$, where $R=R_1\oplus R_2$ is a direct sum of two adjoint representations.
Let $L$ be the standard two dimensional $\sll_2$-module. Let us calculate $P^2(L)$ in $\gmh$. We observe that
$$(R\otimes L)^{sh}\simeq L\oplus L,\quad (R\otimes(R\otimes L)^{sh})^{sh}\simeq L\oplus L\oplus L\oplus L.$$
For any $(x_1,x_2),(y_1,y_2)\in R$ and $v\in L$ we have
$$\pi((x_1,x_2),(y_1,y_2),v)=(x_1y_1v,x_1y_2v,x_2y_1v,x_2y_2v),$$
and
$$\mu((x_1,x_2),(y_1,y_2),v)=([x_1,y_1]v,x_1y_2v-y_1x_2v,x_2y_1v-y_2x_1v,[x_2,y_2]v).$$
One can check that $\operatorname{Coker}\mu=0$ and hence $P^2(L)=0$.

Now consider the universal central extension $\gh$ of $\g$. Then we have $\gh=\g\oplus\kk$ as a vector space, $R^2=\kk$ and
$$\delta((x_1,x_2),(y_1,y_2))=tr x_1y_2-tr x_2y_1.$$
Then $R^2\otimes L=L$ and
$$\lambda((x_1,x_2),(y_1,y_2),v))=(tr x_1y_2-tr x_2y_1)v.$$
Note that for $A,B\in \sll_2$ and $v\in V$ we have 
$$ABv+BAv=(2 tr AB)v.$$
That implies $P^2(L)=\operatorname{Coker}(\lambda\oplus\mu)=L$.

The above example illustrates that $\gmh$ and $\gh\text{-mod}_{\frac12}$ are not equivalent. 
To construct a similar example for the categories $\ggm$ and $\gm$ 
consider the Lie algebra $\g\oplus \sll_2$ and $P(L\boxtimes V)$, where $V$ is the standard module over the second copy of 
$\sll_2$ and $\boxtimes$ stands for the exterior tensor product. 
\end{ex}

\begin{lem}\label{special} Assume that $[\rr,\rr]=0$. Then for any $L\in \gmh$ we have $P^2(L)=0$. 
\end{lem}
\begin{proof} We use the fact that $P^2(L)$ is a $\g_{ss}$-submodule  of $S^2(R)\otimes L$. 
Since we have 
$(\rr_1\otimes L_{\frac12})^{sh}=(\rr_{-1}\otimes L_{-\frac12})^{sh}=0$, $P^2(L)_{-\frac12}$ is in fact a submodule in 
$$M=\left(S^2\rr_0\otimes L_{-\frac12}\right)\oplus \left(\rr_0\otimes \rr_{-1}\otimes L_{\frac12}\right).$$  
By our assumption on $\g$ there are no $(\g_{ss})_{-1}$-invariant vectors in $\rr_0$. Therefore $M$ also does not 
have $(\g_{ss})_{-1}$-invariant vectors. Hence $P^2(L)=0$.
\end{proof}

\subsection {Ext quivers $\JJ$ and $\gm$}
Now let $\J:=Lie(\g)$. Consider the category $\JJ$ and recall the functor $Jor:\gm\to\JJ$. 
If $L\in \gm$ is simple and not trivial, then $Jor(L)$ is simple in $\JJ$ and $Jor(P(L))=P(Jor(L))$ by
Lemma \ref{proj} and Lemma \ref{identity}.

\begin{lem}\label{idempotent} Let 
$$P(\J)=\bigoplus_{L\neq tr} P(Jor(L))$$ 
and $A(\J)=\operatorname{End}_{\J}(P(\J))$.
Then $$A(\J)=(1-e_{tr})A(\gh)(1-e_{tr}),$$
where $e_{tr}$ is the idempotent corresponding to the projector onto $P(tr)$.
\end{lem}
\begin{proof} Follows immediately from Lemma \ref{proj} and the identity
$$Lie(P(\J))=(1-e_{tr})P,$$
where $P$ is the direct sum of all up to isomorphism indecomposable projectives in $\gm$.
\end{proof}

\begin{Cor}\label{Q-Q'-relation} Let $Q(\J)$ be the Ext quiver of the category $\JJ$ and $Q'(\g)$ be the quiver 
obtained from $Q(\g)$ by removing the vertex corresponding to the 
trivial representation. Then $Q'(\g)$ is obtained from $Q(\J)$ by removing some edges.
\end{Cor}
\begin{proof} In notations of the previous proof we have
$$Jor(P^1)\subset rad P(\J)/rad^2 P(\J).$$
Hence the statement.
\end{proof}

\begin{Cor}\label{quiverab} Let $\gh=\g$, the radical $\rr=R$ is abelian and simple over $\g_{ss}$. Then $Q(\J)=Q'(\g)$.
\end{Cor}
\begin{proof} We have to check that for any non-trivial simple $L\in\gm$ we have
$$Jor (P^1(L))=rad(Jor(P(L)))/rad^2(Jor(P(L))).$$
As we already mentioned in the previous corollary we have 
$$Jor (P^1(L))\subset rad(Jor(P(L)))/rad^2(Jor(P(L))).$$
If the inclusion is strict, then by Corollary \ref{Jorradical} $P^1(L)$ contains 
a trivial submodule and $P(L)$ has an indecomposable quotient $M$ of length $3$ such that
$$M/rad M=L,\  rad M/ rad^2 M=tr,\  rad ^2M=L',$$
where $L'$ is some irreducible $\g$-module.
Consider the decomposition $M=L\oplus tr\oplus L'$ over $\g_{ss}$. Since the action $R\otimes M\to M$ is $\g_{ss}$-invariant and $R(rad^i M)\subset rad^{i+1} M$,
we have $L\simeq R^*, L'\simeq R$ and for any  $x\in R, a\in L, b\in \kk, c\in L'$ 
$$x(a,b,c)=(0,t_1\langle x,a\rangle,t_2 bx)$$ 
for some $t_1,t_2\in\kk$.
By obvious calculation $xy\neq yx$ if $x$ and $y$ are not proportional. Hence there is no such module. 
\end{proof}

\section{Applying $Jor$ and $Lie$ to the case of Jordan algebras of bilinear form}

Let $E$ be a finite-dimensional ${\bf k}$-vector space of dimension greater or equal $2$
and $q$ be a symmetric bilinear form on $E$. Then a {\it Jordan algebra of a bilinear form}
$\J=\J(E,q)$ is a vector space $E\oplus{\kk}$ endowed with a multiplication $\circ$
$$
(e_1+\lambda_1)\circ(e_2+\lambda_2)=\lambda_1\lambda_2+ q(e_1,e_2)+\lambda_1 e_2+\lambda_2 e_1,
$$
$e_1,e_2\in E$, $\lambda_1,\lambda_2\in\kk$.
In what follows we assume that $q$ is non-degenerate and consequently $\J(E,q)$ is a simple Jordan algebra.
It is useful to notice that $\J(E,q)$ is a Jordan subalgebra in the Clifford algebra $C(E,q)$ generated
by $E\subset C(E,q)$. If $\operatorname{dim} E$ is even, then $C(E,q)\simeq \operatorname{End}_\kk (S)$, and
$S$ is a unique up to isomorphism special irreducible $\J$-module.  If $\operatorname{dim} E$ is odd, then 
$C(E,q)\simeq \operatorname{End}_\kk (S^+)\oplus  \operatorname{End}_\kk (S^-)$, and $\J$ has two simple special
modules $S^+$ and $S^-$.

We proceed to describing $\g=Lie(\J)$.
Let $V$ be a $n$-dimensional vector space equipped with
non-degenerate symmetric bilinear form  $(\cdot,\cdot)$. The orthogonal Lie
algebra $\g=\so_n$ is the algebra of endomorphisms $A:V\to V$ satisfying
$(Aw,v)+(v,Aw)=0$ for all $w,v\in V$. 
Let $F\subset V$ be a subspace of codimension $2$ such that  $(\cdot,\cdot)$ is non-degenerate on $F$. 
Choose a basis $\xi,\eta\in F^\perp$ such that $(\xi,\eta)=1$, $(\xi,\xi)=(\eta,\eta)=0$. 
Let $h\in \g$ such that $h(\xi)=\xi, \, h(\eta)=-\eta,\, h(F)=0$. Then $h$ defines a short
$\mathbb Z$-grading of $\g$ such that
$$\begin{array}{c}\g_0=\{A\in\g\,|\,A(F)\subset F\},\quad
\g_1=\{A\in\g\,|\,A(\eta)\in F, A(F)\subset \kk \xi\},\\
\g_{-1}=\{A\in\g\,|\,A(\xi)\in F, A(F)\subset \kk \eta\}.
\end{array}
$$ 
Any non-zero element $f\in \g_1$ defines a Jordan algebra structure on $\g_{-1}$ isomorphic to $\J$.
In this way $n=\operatorname{dim}E+3$. 

Next we describe simple objects in $\ggm$ and $\gmh$. 
This description is slightly different in even and odd case. Let $n=2m$ or $2m+1$, 
$\omega_1,\dots \omega_m$ denote the fundamental weights.
We denote by $\Gamma$ the spinor representation of $\so_n$ with highest weight $\omega_m$ for $n=2m+1$ and by 
$\Gamma^\pm$ the 
spinor representations with highest weights $\omega_{m-1}$ and $\omega_m$ for $n=2m$, see  Section 20.1 in \cite{FH} for details.
Other irreducible fundamental representations of $\so_n$ can be obtained by taking the exterior powers of the standard
representation $V$. If $n=2m+1$ they are $\Lambda^i V$ for $i=1,\dots m-1$ with highest weights 
$\omega_1,\dots,\omega_{m-1}$ respectively. Note that $\Lambda^m V$ is irreducible
with highest weight $2\omega_m$. If  $n=2m$, then $\Lambda^i V$ for $i=1,\dots m-2$ are fundamental representations,
$\Lambda^{m-1}V$ is irreducible with highest weight $\omega_{m-1}+\omega_m$. Finally,
$\Lambda^{m}V$ splits into direct sum of two simple modules $\Lambda^+ V \oplus \Lambda^- V$ with highest weights
$2\omega_{m-1}$ and $2\omega_m$ respectively.

\begin{lem}\label{unital} Let $\g=\so_n$.

(a) Assume that $n=2m+1$.
Any simple object in $\gmh$ is isomorphic to the spinor representation $\Gamma$ and 
any simple object in $\ggm$ is isomorphic to $\Lambda^i V$ for some
$1\leq i\leq m$.

(b) Assume that $n=2m$. Any simple object in $\gmh$ is isomorphic to $\Gamma^+$ or $\Gamma^-$.
 Any simple object in $\ggm$ is isomorphic to $\Lambda^i V$ for $i\leq m-1$ or $\Lambda^\pm V$.
\end{lem}

\begin{proof} We will prove (b) leaving (a) to the reader.
After suitable choice of a Cartan subalgebra
and Chevalley generators in $\g$ we may assume that $\mu(h)=(\mu,\omega_1)$ for any weight $\mu$.
Let $\mu$ be the highest weight of a simple $\g$-module $L$.
If $L$ has grading of length $2$, then $(\mu,\omega_1)=\frac12$. If
$L$ has grading of length $3$, then $(\mu,\omega_1)=1$. Note that
$(\omega_i,\omega_1)=1$ for $i\leq m-2$ and
$(\omega_{m-1},\omega_1)=(\omega_{m},\omega_1)=\frac12$. Therefore
if $L$ has grading of length $2$, $\mu=\omega_{m-1}$ or $\omega_m$.
If $L$ has grading of length $3$, we have the following
possibilities:
\begin{itemize}
\item $\mu=\omega_i$ for $i\leq m-2$ and $L\simeq \Lambda^iV$,
\item $\mu=\omega_{m-1}+\omega_m$ and $L\simeq \Lambda^{k-1}V$,
\item $\mu=2\omega_{m-1}$ or $2\omega_m$ and $L\simeq\Lambda^{\pm}V$.
\end{itemize}
\end{proof}
\begin{rem}\label{duality} In the case when $\g=\so_{2m+1}$ all simple modules are self-dual.
If $\g=\so_{2m}$ then $\Lambda^i V\simeq \Lambda^i V^*$, while
$$
(\Gamma^{\pm})^*=\begin{cases}\Gamma^{\pm} \text{\ if\ } m {\text\ is\ even,}\\ 
\Gamma^{\mp} \text{\ if\ } m {\text\ is\ odd,}\\\end{cases} \qquad
(\Lambda^{\pm}V)^*=\begin{cases}\Lambda^{\pm}V \text{\ if\ } m {\text\ is\ even,}\\ 
\Lambda^{\mp}V {\text\ if\ } m \text{\ is\ odd.}\\\end{cases}
$$
\end{rem}

\begin{rem}\label{involution-tau}
An orthogonal Lie algebra $\so_{2m}$ has an involution $\tau$ induced by the symmetry of its Dynkin diagram. It swaps
$\omega_{m-1}$ and $\omega_m$, and therefore $\tau(\Gamma^{\pm})=\Gamma^{\mp}$, 
$\tau(\Lambda^r V)=\Lambda^r V$ and $\tau(\Lambda^{\pm})V=\Lambda^{\mp}V$. 
\end{rem}

Next we calculate $(M\otimes N)^{sh}$ for simple $M$, $N\in \mathcal S$, when $\g=\so_{n}$.
\begin{lem}\label{lemma-tensor} Let $\g=\so_n$ with $n=2m$ or $2m+1$.
\begin{enumerate}
\item\label{tensor-2m+1} For any 
$1\leq r\leq q\leq m$, 
$$(\Lambda^q V\otimes\Lambda^r V)^{sh}=\bigoplus_{i=0}^r\Lambda^{q-r+2i} V.$$ 
\item\label{tensor-2m}  
If $n=2m$, then for any $1\leq r\leq m-1$
$$
(\Lambda^{\pm} V\otimes\Lambda^r V)^{sh}=\begin{cases} \Lambda^{\pm} V \oplus 
\Lambda^{m-2} V\oplus\dots\oplus \Lambda^{m-r} V \qquad {\rm if\ } r \ {\rm is\ even},\\
\Lambda^{m-1} V\oplus \Lambda^{m-3} V\oplus\dots\oplus\Lambda^{m-r} V\quad {\rm if\ } r \ {\rm is\ odd.}
\end{cases}
$$
\item\label{tensor_odd_even} Suppose $n=2m$, $1\leq r\leq m$, then 
$$
(\Gamma^{\pm}\otimes\Lambda^r V)^{sh} =\begin{cases}\Gamma^{\pm}, \  {\rm if\ } r {\rm\ is\ even},\\
\Gamma^{\mp}, \  {\rm if\ } r {\rm\ is\ odd,}\end{cases}\quad
(\Gamma^{\pm}\otimes\Lambda^{\pm})^{sh} =\begin{cases}\Gamma^{\pm}, \  {\rm if\ } m {\rm\ is\ even},\\
0, \  {\rm if\ } m {\rm\ is\ odd}.\end{cases}
$$

\item\label{tensor_gamma} If $n=2m$, then 
$$(\Gamma^{\pm}\otimes \Gamma^{\pm} )^{sh}=(\Gamma^{\pm})^{\otimes 2} =\Lambda^{\pm} V\oplus\bigoplus_{i=1}^{\lfloor\frac{m}{2}\rfloor}\Lambda^{m-2i}V,$$
$$(\Gamma^{+}\otimes\Gamma^{-})^{sh}=\Gamma^{+}\otimes\Gamma^{-}=\bigoplus_{i=1}^{\lfloor\frac{m}{2}\rfloor}\Lambda^{m-2i+1}V$$
\item If $n=2m+1$, then
$$(\Gamma\otimes\Gamma)^{sh}=\Gamma^{\otimes 2} =\bigoplus_{i=0}^m\Lambda^{i}V.$$
\end{enumerate}
\end{lem}
\begin{proof} The formulas for tensor products are given in \cite{Vin}, table\,5, applying $\ ^{sh}$ is straightforward. \end{proof}

\section {Admissible quivers}

We call the quiver $Q(\g)$ {\it admissible} if the associative algebra $\kk Q'(\g)/rad^2$ is tame. That happens exactly when the 
double quiver of $Q'(\g)$
is tame, see Theorem\,\ref{Gabriel}. Let $\J$ be a unital Jordan algebra and $\g=Lie(\J)$.
Lemma\,\ref{idempotent} and Corollary\,\ref{Q-Q'-relation} imply that if $A(J)$ is tame, then $Q(\g)$ is admissible. 
Therefore the first step towards classification of 
tame $A(J)$ is to classify admissible $Q(\g)$.

For the rest of this paper $J$ will be a unital Jordan algebra such that $J_{ss}$ is a direct sum of 
Jordan algebras $J(E,q)$, where $q$ is non-degenerate and $\dim E\geq 4$, $\g$ is the Lie algebra obtained from $J$ by the
Tits-Kantor-Koecher construction.

In this section we classify indecomposable Lie algebras $\g$ with admissible quivers $Q(\g)$ such that $\g_{ss}$ is a 
direct sum of
$\so_n$ with $n\geq 7$. If $\g_{ss}=\so_{n_1}\oplus\so_{n_2}$, then $V$ and $W$ denote the standard representations 
of $\so_{n_1}$ and $\so_{n_2}$ respectively.

\begin{thm}\label{main-quiver}
Let $\g=Lie(\J)$, where  $\J$ is a unital indecomposable Jordan algebra, such that $\J_{ss}$ is a direct sum 
of Jordan algebras of bilinear form over vector space of dimension greater or equal than $4$ and $\rr\neq 0$.
If $Q(\g)$ is admissible, then $Q(\g)$ is one of the following quivers:

\begin{equation}\label{quiver-standard-odd} {\bf Q^{2m+1}_1}: \quad \xymatrix{tr \ar@{-->}@/^0.4pc/[r]^{\gamma_0}&
V\ar@/^0.4pc/[r]^{\gamma_1}\ar@{-->}@/^0.4pc/[l]^{\delta_0} & \Lambda^2 V \ar@/^0.4pc/[r]
^{\gamma_2} \ar@/^0.4pc/[l]^{\delta_1}&\dots 
\ar@/^0.4pc/[l]^{\delta_2}\ar@/^0.4pc/[r] & \Lambda^{m-1} 
V\ar@/^0.4pc/[r]^{\gamma_{m-1}}\ar@/^0.4pc/[l]& 
\Lambda^m V\ar@/^0.4pc/[l]^{\delta_{m-1}}\ar@{->}^{\gamma_{m}}@(ur,dr)}
\end{equation}

\begin{equation}\label{quiver-standard-even} 
\xymatrix{ &&&&&& \Lambda^+\ar@/^0.4pc/[dl]^{\delta^+}\\
{\bf Q^{2m}_1}:&tr\ar@{-->}@/^0.4pc/[r]^{\gamma_0}  & V\ar@/^0.4pc/[r]^{\gamma_1} \ar@{-->}@/^0.4pc/[l]^{\delta_0}& 
\Lambda^2 V\ar@/^0.4pc/[r]\ar@/^0.4pc/[l]^{\delta_1} & \dots \ar@/^0.4pc/[r]^{\gamma_{m-2}\ }\ar@/^0.4pc/[l] &
\Lambda^{m-1} V\ar@/^0.4pc/[l]^{\delta_{m-2}\ }\ar@/^0.4pc/[ur]^{\gamma^+}\ar@/^0.4pc/[dr]^{\gamma^-}&\\
&&&&&&\Lambda^-\ar@/^0.4pc/[ul]^{\delta^-}}
\end{equation}

\begin{equation}\label{quiver-lambda-plus}  {\bf Q_2}: \qquad \xymatrix {V\ar@/^0.4pc/[r]^{\alpha_1} & \Lambda^3 V \ar@/^0.4pc/[l]^{\beta_1} 
\ar@{<-}_{\gamma_1}@(dr,ur)&  &
\Lambda^2 V\ar@/^0.4pc/[r]^{\alpha_2}\ar@{<-}^{\gamma_2}@(dl,ul) & \Lambda^+ \ar@{->}^{\gamma_3}@(ul,ur)\ar@/^0.4pc/[l]^{\beta_2}\ar@{-->}@/^0.4pc/[r] 
& tr\ar@{-->}@/^0.4pc/[l]  & \Lambda^-}
\end{equation}

\begin{equation}\label{so_8-so_8}  \xymatrix{& \Lambda^2 V\ar@/^0.4pc/[dr]^{\alpha_1}
& tr\ar@{-->}@/^0.3pc/[d] & {\Lambda^2 W}\ar@/^0.4pc/[dl]^{\alpha_3}\\
{\bf Q_3} &  & {\Gamma^+_1\boxtimes\Gamma^+_2}\ar@{-->}@/^0.3pc/[u]
\ar@/^0.4pc/[ur]^{\beta_3}\ar@/^0.4pc/[ul]^{\beta_1}\ar@/^0.4pc/[dr]^{\beta_4}\ar@/^0.4pc/[dl]^{\beta_2} & \\
 &{\Lambda^+ V}\ar@/^0.4pc/[ur]^{\alpha_2} & & {\Lambda^+ W}\ar@/^0.4pc/[ul]^{\alpha_4}} 
\qquad \quad \xymatrix{\Lambda^3 V\ar@/^0.4pc/[d]^{\gamma_1}\\ 
{\Gamma^-_1\boxtimes\Gamma^+_2}\ar@/^0.4pc/[d]^{\delta_2} \ar@/^0.4pc/[u]^{\delta_1}\\ V\ar@/^0.4pc/[u]^{\gamma_2}} 
\qquad \quad \xymatrix{\Lambda^3 W\ar@/^0.4pc/[d]^{\gamma_3}\\ 
{\Gamma^+_1\boxtimes\Gamma^-_2}\ar@/^0.4pc/[d]^{\delta_4} \ar@/^0.4pc/[u]^{\delta_3}\\ W\ar@/^0.4pc/[u]^{\gamma_4}} \qquad \xymatrix{\Lambda^- V\\ 
\Lambda^-W\\ \Gamma_1^-\boxtimes\Gamma_2^-}
\end{equation}

\begin{equation}\label{so_8-so_10} 
\xymatrix{&& tr \ar@{-->}[dddl] &&\\
{\bf Q_4}&& W \ar[ddr]_{\delta_1} &&\\
&& \Lambda^3 W \ar[dr]_{\delta_2} &&\\
\Lambda^- W & {\Gamma_1^+\boxtimes\Gamma_2^-}\ar[l]^{\gamma_3}\ar[uur]_{\gamma_1}\ar[ur]_{\gamma_2} & \Lambda^2 W \ar[l]_{\quad\beta_1} 
& {\Gamma_1^+\boxtimes\Gamma_2^+}\ar@{-->}[uuul] \ar[dl]_{\alpha_2}\ar[ddl]_{\qquad\alpha_3}\ar[dddl]^{\alpha_4\quad} \ar[l]_{\alpha_1} & \Lambda^+ W \ar[l]_{\quad\delta_3}\\
&& \Lambda^4 W \ar[ul]_{\beta_2} &&\\
&& \Lambda^2 V \ar[uul]_{\beta_3} &&\\
&& \Lambda^+ V \ar[uuul]^{\beta_4} &&\\} \quad \xymatrix{ & &\\ &{\Gamma_1^-\boxtimes \Gamma_2^-}\ar[dr]^{\rho_1}  \ar[dl]_{\rho_2} &\\
V\ar[dr]_{\tau_1}& &{\Lambda^3 V}\ar[dl]^{\tau_2}\\ &\Gamma_1^-\boxtimes\Gamma_2^+&  & &\\ \\  &\Lambda^- V &}
\end{equation}

\begin{equation}\label{so_10-so_10} 
\xymatrix{&& \Lambda^3 V \ar[dddl]_{\beta_2} &&\\
{\bf Q_5}&& W \ar[ddr]_{\delta_1} &&\\
\Lambda^+ V\ar[dr]^{\beta_3}&& \Lambda^3 W \ar[dr]_{\delta_2} && \Lambda^- V\\
\Lambda^- W  & {\Gamma_1^+\boxtimes\Gamma_2^-}\ar[l]_{\alpha_3} \ar[ur]_{\alpha_2} \ar[uur]_{\alpha_1}  & V \ar[l]_{\quad\beta_1} 
& {\Gamma_1^-\boxtimes\Gamma_2^+}\ar[ur]^{\gamma_3}\ar[uuul]_{\gamma_2}\ar[l]_{\gamma_1\quad}   & \Lambda^+ W\ar[l]_{\quad\delta_3}\\}
\xymatrix{ & \Lambda^2 V\ar[dl]_{\rho_1} &\\{\Gamma_1^-\boxtimes \Gamma_2^-} &\Lambda^4 V\ar[l]_{\quad\rho_2} & \Gamma_1^+\boxtimes\Gamma_2^+\ar[dl]_{\tau_3} 
\ar[l]_{\tau_2} \ar[ul]_{\tau_1} 
\ar[ddl]_{\tau_4} \ar@{-->}[dddl]\\
& \Lambda^2 W\ar[ul]_{\rho_3} & \\ & \Lambda^4 W\ar[uul]_{\rho_4}  &\\ &tr\ar@{-->}[uuul]&}
\end{equation}

Indecomposable Lie algebras $\g=\gb=\g_{ss}\oplus R$ with admissible $Q(\g)$  are  listed below. We will further refer to this list as {\bf List A} :
\begin{enumerate}
\item ${\g}=\so_{2m+1}\niplus V$, $m\geq 3$, $Q(\g)={\bf Q_{2m+1}^1}$;
\item ${\g}=\so_{2m}\niplus V$, $m\geq 4$, $Q(\g)={\bf Q_{2m}^1}$;
\item ${\g}=\so_8\niplus\Lambda^{\pm}$; $Q(\so_8+\Lambda^+)={\bf Q_2}$, while $Q(\so_8\niplus\Lambda^-)$ is obtained by 
application of $\tau$ to ${\bf Q_2}$, see Remark\,\ref{involution-tau}.
\item ${\g}=(\so_8\oplus\so_8)\niplus\Gamma_1^{\pm}\boxtimes\Gamma_2^{\pm}$; $Q((\so_8\oplus\so_8)\niplus\Gamma_1^+\boxtimes\Gamma_2^+)={\bf Q_3}$, 
the quivers corresponding to $\Gamma_1^+\boxtimes\Gamma_2^-$,
($\Gamma_1^{-}\boxtimes\Gamma_2^{+}$ and $\Gamma_1^{-}\boxtimes\Gamma_2^{-}$) are obtained from ${\bf Q_3}$ applying $1\times\tau$ (respectively
$\tau\times 1$ and $\tau\times\tau$). 
\item ${\g}=(\so_8\oplus\so_{10})\niplus\Gamma_1^{\pm}\boxtimes\Gamma_2^{\pm}$, 
$Q\left((\so_8\oplus\so_{10})\niplus\Gamma_1^{+}\boxtimes\Gamma_2^{+}\right)={\bf Q_4}$, while other 
quivers are obtained by application of $1\times\tau$, $\tau\times 1$ and $\tau\times\tau$ to ${\bf Q_4}$.
\item ${\g}=(\so_{10}\oplus\so_{10})\niplus\Gamma_1^{\pm}\boxtimes\Gamma_2^{\pm}$, $Q\left((\so_{10}\oplus\so_{10})\niplus\Gamma_1^+\boxtimes\Gamma_2^+\right)={\bf Q_5}$, 
while other quivers are obtained by application of $1\times\tau$, $\tau\times 1$ and $\tau\times\tau$ to ${\bf Q_5}$.
\end{enumerate}
\end{thm}

\begin{proof} 
Suppose $J$ satisfies the conditions of the theorem, then $\g=Lie(\J)=\g_{ss}\niplus\rr$, where $\g_{ss}$ is a direct sum of orthogonal 
algebras $\so_n$, $n\geq 7$. 
Since $Q(\g)=Q(\gb)$ we may assume that $\g=\gb$ and hence $\rr=R$.
To construct $Q(\g)$  we use Lemma\,\ref{quiver_Lie}.
We start with classifying admissible quivers $Q(\g)$ in the case when $R$ is an irreducible faithful $\g_{ss}$-module.

Consider first the case $\g=\so_{2m+1}\niplus R$, $m\geq 3$. There are $m+1$ simple modules in the category $\gm$, namely $tr$ and 
$\Lambda^r V$, $r=1,\dots,m$, thus $Q(\g)$ has $m+1$ vertices. Let $R=V$ be the standard representation of $\so_{2m+1}$. 
Tensor product formulas in Lemma \ref{lemma-tensor}\eqref{tensor-2m+1} imply that the quiver of $\so_{2m+1}\niplus V$ is ${\bf Q_1^{2m+1}}$.
It is admissible by Theorem \ref{Gabriel}\eqref{double-quiver}.

Next we claim that if $R=\Lambda^r V$, $r\geq 2$ then the quiver $Q(\g)$ is not admissible. Indeed,
Lemma\,\ref{lemma-tensor}\eqref{tensor-2m+1} implies 
that $\Lambda^{m} V$ and $\Lambda^{m-1}V$ are simple constituents of $\Lambda^{m} V\otimes\Lambda^{r} V$,
$\Lambda^{m} V,\Lambda^{m-1}V, \Lambda^{m-2} V$ are simple  constituents of $\Lambda^{m-1} V\otimes\Lambda^{r} V$,
while  $\Lambda^{m-1} V$ and $\Lambda^{m-2}V$ are simple  constituents of
$\Lambda^{m-2} V\otimes\Lambda^{r} V$.  Therefore $Q(\g)$ 
has the following subquiver

\begin{equation}
\xymatrix{\Lambda^{m-2} V\ar@/^0.4pc/[r] & \Lambda^{m-1} V \ar@/^0.4pc/[r] 
\ar@/^0.4pc/[l]\ar@{->}@(ul,ur)& 
\Lambda^m V\ar@/^0.4pc/[l] \ar@/^0.4pc/[l]\ar@{->}@(ur,dr)}
\end{equation}

The corresponding double quiver is wild by Theorem\,\ref{Gabriel}\eqref{double-quiver}, hence the double quiver of $Q'(\g)$ is wild 
by Lemma \ref{lemma-pro-subquiver}. 
Therefore  $Q(\g)$ is not admissible.

Next, let us consider the case $\g=\so_{2m}\niplus R$, with $m\geq 4$. All up to isomorphism simple objects of $\ggm$ are
$tr$, $\Lambda^r V$, $r=1,\dots,m-1$ and $\Lambda^+ V$, $\Lambda^- V$. By Lemma\,\ref{lemma-tensor} 
$Q(\so_{2m}\niplus V)={\bf Q_1^{2m}}$. It is again admissible by Theorem \ref{Gabriel}\eqref{double-quiver}.

Let $R=\Lambda^r V$, $r=2,\dots,m-1$. We will show that $Q(\g)$ is not admissible. Lemma\,\ref{lemma-tensor}\eqref{tensor-2m+1} 
implies that if $m$ is even, then $\Lambda^{m-1}V\otimes\Lambda^r V$ contains $\Lambda^{m-1} V$ with multiplicity $2$ and 
$\Lambda^{m-3} V$ with multiplicity $1$. Hence $Q(\g)$ has the subquiver

\vspace{.5cm}
\begin{equation}
\xymatrix{\Lambda^{m-3} V & \Lambda^{m-1} V \ar[l]\ar@{->}@(ul,ur)\ar@{->}@(dl,dr)}
\end{equation}
\vspace{.5cm}

Thus, $Q(\g)$ is not admissible by Theorem \ref{Gabriel} and  Lemma \ref{lemma-pro-subquiver}.

Similarly, if $m$ is odd, $\Lambda^{m-1}V\otimes\Lambda^r V$ contains $\Lambda^{m-2} V$ with multiplicity $2$ and 
$\Lambda^+V$ with multiplicity $1$. 
Therefore $Q(\g)$ has the wild subquiver 
\begin{equation}
\xymatrix{\Lambda^{m-2} V & \Lambda^{m-1} V \ar@/^0.2pc/[l] \ar@/_0.2pc/[l] \ar[r] &
\Lambda^{+} V}
\end{equation} 
Thus,  $Q(\g)$ is not admissible by Lemma \ref{lemma-pro-subquiver}.

Let $R=\Lambda^{\pm} V$. For $\g=\so_8\niplus\Lambda^{+}V$  the Ext quiver of $\gm$ is ${\bf Q_2}$, which is admissible. 
The same applies to $Q(\g)$ for $\g=\so_8\niplus\Lambda^-V$, 
since the involution $\tau$ interchanges the vertices $\Lambda^+V$ and $\Lambda^-V$ of  ${\bf Q_2}$. 

By Lemma\,\ref{lemma-tensor}\eqref{tensor-2m} we obtain that $Q'(\g)$, where $\g=\so_{10}\niplus\Lambda^+V$, has the subquiver

\bigskip
\begin{equation}
\xymatrix{ &&&& \Lambda^+V\\
\so_{10}+\Lambda^+:&V\ar@/^1.4pc/[rrr] & \Lambda^2 V \ar@/^0.4pc/[urr] \ar@/^0.4pc/[r]& \Lambda^3 V \ar@/^0.4pc/[r] 
\ar@/^0.4pc/[l]&
\Lambda^4 V\ar@/^1.4pc/[lll]\ar@/^0.4pc/[l]\ar[u]&\\
&&&&\Lambda^-V\ar@/^0.4pc/[ull]\ar[u]}
\end{equation}
\bigskip

\noindent By Theorem\, \ref{Gabriel}\eqref{double-quiver} the double quiver of the above quiver is wild. Hence $Q(\g)$ is not admissible. 
The same argument works for $R=\Lambda^-V$.

For $m\geq 6$ one of the following subquivers  

\begin{equation}
\xymatrix{\Lambda^{m-4} V\ar@/^0.4pc/[r] & \Lambda^{m-2} V \ar@/^0.4pc/[r] 
\ar@/^0.4pc/[l]\ar@{->}@(ul,ur)& 
\Lambda^+V\ar@/^0.4pc/[l] \ar@/^0.4pc/[l]\ar@{->}@(ur,dr)} \qquad \quad {\rm if\ } m \ {\rm is\ even} 
\end{equation}

\begin{equation}
\xymatrix{ \Lambda^{m-6} V& \Lambda^+& \Lambda^{m-3} V\ar[l]&\\
 \Lambda^{m-4} V&\Lambda^{m-1\ar[l]\ar[r]\ar[u]\ar[ul]} V& \Lambda^{m-2} V & {\rm if\ } m \ {\rm is\ odd}}
\end{equation}

\noindent is a subquiver of $Q'(\so_{2m}+\Lambda^+ V)$. Both have wild double quivers and hence are not admissible. 

Next we move to the case when
$\g_{ss}=\so_{m_1}\oplus\so_{m_2}$, $m_1,\,m_2\geq 7$, and $R$ is an exterior tensor product of spinor modules 
$\Gamma$ or $\Gamma^{\pm}$, depending on the parity of $m_1$ and $m_2$, see Section 5 for details.

First, assume that $m_1=2m+1$ and $m_2=2n+1$ are odd, $m,n\geq 3$. Denote by $\Gamma_i$, $i=1,2$ the spinor representations of $\so_{m_i}$. 
Let $\g=(\so_{2m+1}\oplus\so_{2n+1})\niplus\Gamma_1\boxtimes\Gamma_2$. 
Lemma\, \ref{lemma-tensor} (5) implies 
$$\left((\Gamma_1\boxtimes\Gamma_2)\otimes (\Gamma_1\boxtimes\Gamma_2)\right)^{sh}=\bigoplus_{i=0}^m\Lambda^{i}V\oplus\bigoplus_{i=0}^n\Lambda^{i}W.$$
Therefore $Q'(\g)$ has a wild subquiver
\begin{equation}
\xymatrix{\Lambda^2 V&\Lambda^3 V&\Lambda^2 W\\
V&\Gamma_1\boxtimes\Gamma_2\ar[l]\ar[r]\ar[ur]\ar[u]\ar[ul]&W}
\end{equation}
and hence $Q(\g)$ is not admissible. 

Next let $m_1=2m+1,\,m_2=2n$, $n\geq 3$, $m\geq 4$ and $\g=(\so_{2m+1}\oplus\so_{2n})\niplus\Gamma_1\boxtimes\Gamma_2^+$. 
By Lemma\,\ref{lemma-tensor} (4) and (5) we easily obtain that 
$\left((\Gamma_1\boxtimes\Gamma_2^{+})\otimes(\Gamma_1\boxtimes(\Gamma_2^+)^{*})\right)^{sh}$ 
has at least five simple constituents: 
$$V,\quad\Lambda^2 V,\quad\Lambda^{3} V,\quad\Lambda^2 W,\quad\Lambda^4 W \quad(\Lambda^+W\quad\text{if}\quad m=4).$$ 
Therefore  the vertex $\Gamma_1\boxtimes(\Gamma_2^+)^{*}$ has at least five outgoing arrows in $Q'(\g)$. 
As in the previous case $Q(\g)$ is not admissible. 
The case of the algebra $R=\Gamma_1\boxtimes\Gamma_2^-$ can be reduced to the previous one by applying $1\times\tau$.

Finally, we have to deal with the case when both $m_1=2m$ and $m_2=2n$ are even, $n\geq m\geq 4$. 
Using Lemma\,\ref{lemma-tensor}\,\eqref{tensor_odd_even},\eqref{tensor_gamma} we obtain that quivers $Q(\g_i)$, $i=3,4,5$ of algebras 
$$\g_3=(\so_{8}\oplus\so_8)\niplus\Gamma_1^+\boxtimes\Gamma_2^+, 
\quad \g_4=(\so_{8}\oplus\so_{10})\niplus\Gamma_1^+\boxtimes\Gamma_2^+ ,
\quad \g_5=(\so_{10}\oplus\so_{10})
\niplus\Gamma_1^+\boxtimes\Gamma_2^+
$$ 
are ${\bf Q_3}$, ${\bf Q_4}$ and 
${\bf Q_5}$ respectively. By direct inspection they are admissible. Furthermore, the same is true
for $R=\Gamma_1^{\pm}\boxtimes\Gamma_2^{\pm}$, by application of
$\tau\times 1$, 
$1\times\tau$ or $\tau\times \tau$.

We claim that if $m\geq 4$, $n\geq 6$ and $R=\Gamma_1^+\boxtimes\Gamma_2^+$, then $Q(\g)$ is not admissible. 
Indeed, we use the same argument as before.  Lemma\,\ref{lemma-tensor} (4) implies that
$((\Gamma_1^+\boxtimes\Gamma_2^+)\otimes((\Gamma_1^+)^*\boxtimes(\Gamma_2^+)^*))^{sh}$ has at least five simple simple constituents: 
$$\Lambda^2 V,\quad \Lambda^4 V\quad(\Lambda^+ V\quad\text{if}\quad m=4),\quad \Lambda^2 W,\quad \Lambda^4 W, \quad\Lambda^6 W \quad(\Lambda^+ W\quad\text{if}\quad n=6).$$ 
There are at least five outgoing arrows in $Q'(\g)$ from the vertex $(\Gamma_1^+)^*\boxtimes(\Gamma_2^+)^*$. Theorem \ref{Gabriel}\eqref{Dynkin}
implies that $Q(\g)$ is not admissible.

We have shown that if $R$ is a faithful irreducible module, $\g$ is indecomposable, 
then $Q(\g_{ss}\niplus R)$ is admissible if and only 
if  $\g$ is one of the algebras from {\bf List A}.
It remains to prove that $Q(\g)$ is not admissible if $R$ is not simple. It follows from the observation that adding an irreducible component to $R$ implies 
adding at least one outgoing arrow to the vertex corresponding to $\Lambda^2V$ or $\Lambda^2W$. 
We leave it to the reader to check that 
adding such an arrow to one of the quivers from the list makes the corresponding double quiver wild.  
\end{proof}
\begin{rem} If $\g$ is from {\bf List A}, then $\gh=\g$ since $(\Lambda^2 R)^{\g_{ss}}=0$.
\end{rem}

\section{Relations in the case of abelian radical}

Let us assume that $\g=Lie(\J)$ is a Lie algebra from {\bf List A}.
The goal of this section is to show that for any such $\g$ the algebra $A(\J)$ is tame.
Recall that by Corollary \ref{quiverab} $A(\J)$ is a quotient of the path algebra $\kk Q'(\g)$ by some ideal $I$. It turns out
that $I$ is generated by quadratic relations and to describe them it is sufficient to calculate $P^2(L)$ for simple $L$ in $\ggm$.
We will see also that $rad^3(A(J))=0$.
 
\subsection{The case of simple $\g_{ss}$} In this subsection we assume that $\g_{ss}$ is simple, i.e. $\g$ is a Lie algebra from (1),(2) or (3) of
{\bf List A}.

\begin{prop}\label{standard1} Let $\g=\so_n\niplus V$. Then the first three layers of the radical filtration of indecomposable projectives in $\ggm$ are as follows
\begin{equation}\label{projective-standard}
\begin{array}{c}
\Lambda^r V\\
\overline{\Lambda^{{r-1}^{\ }}V\oplus\Lambda^{r+1}V}\\
\overline{\quad \Lambda^r V\quad }\\
{\text if \ } n=2m+1,\ 0\leq r\leq m\\
{\text if \ } n=2m,\ 0\leq r\leq m-2\\
\end{array} \qquad
\begin{array}{c}
\Lambda^{m-1} V\\
\overline{\Lambda^{m-2} V\oplus\Lambda^+ V\oplus\Lambda^{-}V }\\
\overline{\quad \Lambda^{m-1} V\quad }\\
{\text if \ } n=2m\\
\\
\end{array} \qquad \quad
\begin{array}{c}
\Lambda^{\pm} V\\
\overline{\quad \Lambda^{m-1} V\quad}\\
\overline{\quad \Lambda^{\mp} V \quad}\\
{\text if \ } n=2m\\
\\
\end{array} 
\end{equation}
\end{prop}
\begin{rem} Note that for an odd $m$ we have $P^1(\Lambda^m V)=\Lambda^{m-1} V\oplus\Lambda^{m} V$ due 
to isomorphism $\Lambda^m V\simeq\Lambda^{m+1}V$. We also assume $\Lambda^{-1}V=0$.
\end{rem}
\begin{proof}
For any $v\in V$ we introduce the following operators $m_{v},i_{v}\in{\rm End}(\Lambda^{\bullet} V)$:
\begin{equation}\label{definition-mv-iv}
\begin{array}{l}
m_v(x_1\wedge\dots\wedge x_r)=v\wedge x_1\wedge\dots\wedge x_r,\\
i_v(x_1\wedge\dots\wedge x_r)=\sum^{i=1}_r\langle v,x_i\rangle 
(-1)^{i-1} x_1\wedge\dots\wedge\widehat{x}_i\wedge\dots\wedge x_r,
\end{array}
\end{equation}
$x=x_1\wedge\dots\wedge x_r\in \Lambda^r V$. These operators satisfy the following well-known relations:
\begin{equation}\label{relations-standard-m-i}
\begin{array}{c}
i_v i_w+i_w i_v=0;\\
m_v m_w+ m_w m_v=0;\\
i_v m_w+m_w i_v=(v,w).
\end{array}
\end{equation}
Moreover, the action of algebra $\so_n\simeq\Lambda^2 V$ on $\Lambda^{\bullet} V$ can be written as 
\begin{equation}\label{adjoint-action-so-2m}
T_{v\wedge w}(x)=m_v i_w(x)-m_w i_v (x), \quad x\in\Lambda^{\bullet} V,\  v\wedge w\in\so_n.
\end{equation}

First assume that $n=2m+1$, then  any simple module $L$ in $\ggm$  is isomorphic to $\Lambda^rV$ for some $0\leq r\leq m$.

By Lemma\,\ref{lemma-tensor} $$P^1(L)=(L\otimes V)^{sh}\simeq \Lambda^{r-1}V\oplus \Lambda ^{r+1} V.$$ 
Note that $\g_{ss}$-invariant maps
$\Lambda^r V\otimes V\to \Lambda^{r+1} V$ and $\Lambda^r V\otimes V\to \Lambda^{r-1} V$  
are given by $x\otimes v\mapsto m_v(x)$ and $x\otimes v\mapsto i_v(x)$ respectively.
To describe $P^2(L)$ we use Lemma \ref{technical2} with $\lambda=0$.
Indeed,
$$(R\otimes (R\otimes L)^{sh})^{sh}=\Lambda^{r+2}V\oplus\Lambda^r V\oplus \Lambda^r V\oplus \Lambda^{r-2}V,$$
and (\ref{relations-standard-m-i}), (\ref{adjoint-action-so-2m}) imply
$$\mu(v,w,x)=(2m_vm_w(x),-T_{v\wedge w}(x),T_{v\wedge w}(x), 2i_vi_w(x)).$$
That implies $P^2(L)\simeq L$.

Now let $n=2m$. In this case the calculation of the radical filtration of $P(L)$ for a simple $L=\Lambda^rV$ for $r\leq m-1$ is the same as in the case of odd $n$.
It remains to consider the cases $L=\Lambda^{\pm}V$.  
Then we have $P^1(L)=(L\otimes V)^{sh}=\Lambda^{m-1}V$.
Recall that we have a decomposition $\Lambda^mV=\Lambda^+V\oplus\Lambda^-V$.
After suitable normalization 
\begin{equation}\label{lambda-plus-or-minus}
\Lambda^{\pm} V=\{x\in \Lambda^{m} V\,|\, i_v(x)=\pm\psi m_v(x),\ \text{for\ all\ } v\in V\},
\end{equation} 
where $\psi:\Lambda^{m+1} V\to \Lambda^{m-1} V$ is an isomorphism of simple $\g_{ss}$-modules.
Furthermore,
$$(R\otimes (R\otimes L)^{sh})^{sh}=\Lambda^{m-2}V\oplus\Lambda^{m-1} V,$$
$$\mu(v,w,x)=(2i_vi_w(x),T_{v\wedge w}(x)).$$
The relation (\ref{adjoint-action-so-2m}) imply $\operatorname{Im}\mu=\Lambda^{m-2}V\oplus L$. Therefore we have $P^2(\Lambda^{\pm}V)=\Lambda^{\mp}V$. 
\end{proof}

\begin{prop}\label{projectives-so-lambda}
Let $\g=\so_8\niplus\Lambda^{+}V$. Then $\Lambda^- V$ is projective and
other indecomposable projectives have the following first three layers in the radical filtration 
\begin{equation}\label{projective-lambda-plus}
\begin{array}{c}
\Lambda^2 V\\
\overline{\Lambda^{2^{\ }}V\oplus\Lambda^+ V}\\
\overline{\quad \Lambda^2 V\quad }\\
\end{array} \qquad
\begin{array}{c}
\Lambda^+ V\\
\overline{\Lambda^{2^{\ }}V\oplus\Lambda^+ V \oplus tr} \\
\overline{\quad \quad \Lambda^+ V\oplus tr\quad \quad }\\
\end{array} \qquad
\begin{array}{c}
\Lambda^3 V\\
\overline{V^{\ }\oplus \Lambda^3 V}\\
\\
\end{array} \qquad \begin{array}{c}
V\\
\overline{\Lambda^3 V}\\
\\
\end{array} 
\end{equation}
\end{prop}
\begin{proof}
It follows from Lemma\,\ref{lemma-tensor} that $\Lambda^- V$ is projective  in $\ggm$.
To describe the projective covers of $\Lambda^2V$ and $\Lambda^+V$ we use an automorphism $\gamma$ of $\so_8$, induced by a rotation of the Dynkin diagram $D_4$.
Twisting by $\gamma$ defines the following identifications on simple modules
$$V\mapsto \Gamma^+,\quad \Lambda^2V\mapsto \Lambda^2V,\quad \Lambda^3V\to (V\otimes \Gamma^-)_0 \quad\Lambda^+\mapsto S^2V_0,$$
where by $S^2V_0$ we denote the traceless part of $S^2V$ and $V\otimes \Gamma^-=(V\otimes \Gamma^-)_0\oplus \Gamma^+$.

Let us calculate $P^2(L)$ for the case $L=S^2V_0$ using Lemma \ref{technical2} with $\lambda=0$. We identify $S^2V$ and $\Lambda^2V$ with the spaces of 
symmetric and skew symmetric matrices respectively.
We have 
$$(R\otimes (R\otimes L)^{sh})^{sh}=\Lambda^2V\oplus\Lambda^2V\oplus S^2V\oplus S^2V_0,\quad (R^2\otimes L)^{sh}=\Lambda^2V\oplus S^2V_0,$$
and
$$\pi(X\otimes Y\otimes A)=\left(\{X,[Y,A]\},[X,\{Y,A\}],\{X,\{Y,A\}\},[X,[Y,A]]\right),$$
where $\{C,B\}=CB+BC$ and $[C,B]=CB-BC$.
Next we calculate $\mu$:
$$\mu(X\otimes Y\otimes A)=\left(\{A,[X,Y]\},\{A,[X,Y]\}+2XAY-2YAX,[A,[X,Y]],[[X,Y],A]\right).$$
From this formula we see that cokernel of $\mu$ is isomorphic to $S^2V_0\oplus tr$.

Now let $L=\Lambda^2V$. Then
$$(R\otimes (R\otimes L)^{sh})^{sh}=\Lambda^2V\oplus\Lambda^2V\oplus S^2V\oplus S^2V_0,\quad (R^2\otimes L)^{sh}=\Lambda^2V\oplus S^2V_0,$$
and 
$$\pi(X\otimes Y\otimes A)=\left(\{X,\{Y,A\}\},[X,[Y,A]],\{X,[Y,A]\},[X,\{Y,A\}]\right),$$
$$\mu(X\otimes Y\otimes A)=\left([A,[X,Y]],[A,[X,Y]],\{A,[X,Y]\},\{A,[X,Y]\}+2XAY-2YAX\right),$$
and cokernel of $\mu$ is isomorphic to $\Lambda^2V$.

%If $L=tr$, then
%$$(R\otimes R)^{sh}=\Lambda^2V\oplus S^2V,\quad, \pi(X\otimes Y)=XY,\quad \mu (X\otimes Y)=[X,Y],$$
%hence cokernel of $\mu$ is isomorphic to $S^2V$.

%The proof that $P^3(\Lambda^2V)=P^3(S^2V_0)=0$ is the same as for the case $\g=\so_n+V$.

Next we construct the projective covers of $V$ and $\Lambda^3 V$ in $\gm$. We will show that both modules have Loewy length two.
Let $\{e_i,\,e_{-i}\,|\,1\leq i\leq 4\}$ be the basis of $V$ such that with respect to the form 
on $V$, $(e_i,e_{-j}) =\delta_{i,-j}$, $i,j\in\{1,2,3,4\}$. 
Then $\Lambda^4 V$ is spanned by $$e_{i_1}\wedge e_{i_2}\wedge e_{i_3}\wedge e_{i_4}, \quad i_1<i_2<i_3<i_4,$$ to check whether given element of $\Lambda^4 V$ belongs to 
$\Lambda^+ V$ we use \eqref{lambda-plus-or-minus}.

From Lemma\,\ref{lemma-tensor} $$P^1(\Lambda^3 V)=(\Lambda^3 V\otimes\Lambda^+ V)^{sh}=V\oplus \Lambda^3 V,\quad  
P^1(V)=(V\otimes\Lambda^+ V)^{sh}=\Lambda^3 V.$$ One can check that  $\so_8$-invariant maps  
$$\phi_3^3:\Lambda^3 V\otimes \Lambda^+ V\to \Lambda^3 V, \quad \phi^3_1:\Lambda^3 V\otimes \Lambda^+ V\to V, \quad
\phi^1_3:V\otimes\Lambda^+ V\to \Lambda^3 V$$  are given by 
\begin{equation}\phi^3_3(x_1\wedge x_2\wedge x_3\otimes v_1\wedge v_2\wedge v_3 \wedge v_4)=\sum_{\begin{array}{c} 1\leq j,k,l\leq 4\\ j\neq k\neq l\end{array}} 
\operatorname{sgn}(j,k,l) m_{x_j} i_{x_k} i_{x_l}(v_1\wedge v_2\wedge v_3 \wedge v_4),
\end{equation} 
\begin{equation} \phi^3_1(x_1\wedge x_2\wedge x_3\otimes v_1\wedge v_2\wedge v_3 \wedge v_4)=\sum_{\begin{array}{c} 1\leq j,k,l\leq 4\\ j\neq k\neq l\end{array}} 
\operatorname{sgn}(j,k,l) i_{x_j} i_{x_k} i_{x_l}(v_1\wedge v_2\wedge v_3 \wedge v_4),
\end{equation} 
\begin{equation}\phi_3^1 (x_1\otimes v_1\wedge v_2\wedge v_3 \wedge v_4)=i_{x_1}(v_1\wedge v_2\wedge v_3 \wedge v_4).\end{equation}
Here $v_1\wedge v_2\wedge v_3 \wedge v_4 \in\Lambda^+ V$, $x_j\in V$, $1\leq j\leq 4$, $\operatorname{sgn}(j,k,l)$ is the 
sign of permutation $(j,k,l)\in\Sigma_3$, 
while $m_{x_j}$ and $i_{x_j}$ are given by \eqref{definition-mv-iv}.

To show that $P^2(V)=0$ we describe map $\mu$ for $L=V$, see Lemma\,\ref{technical2}. Observe that $\lambda=0$
therefore $P^2(V)=\operatorname{Coker}\mu$. We have 
$$
(\Lambda^+ V\otimes(\Lambda^+ V\otimes V)^{sh})^{sh}=V\oplus\Lambda^3 V,
$$
hence 
$$
\mu(v,w,x)=(\phi_3^1(\phi_1^3(x\otimes v)\otimes w)-\phi_3^1(\phi_1^3(x\otimes w)\otimes v),\phi_3^3(\phi_1^3(x\otimes v)\otimes w)-\phi_3^3(\phi_1^3(x\otimes w)\otimes v)),
$$
where $v,w\in\Lambda^+ V$, $x\in V$.
Suppose $v=e_1\wedge e_2\wedge e_3 \wedge e_4$, $w=e_{-1}\wedge e_{-2}\wedge e_{-3}\wedge e_{-4}$ and $x=e_1$, then
$$\mu(v,w,x)=(-e_1,e_1\wedge e_2\wedge e_{-2}+e_1\wedge e_3\wedge e_{-3}+e_1\wedge e_4\wedge e_{-4}).$$ 
Since $\operatorname{Im}\mu$ is $\g_{ss}$-invariant and $\mu(v,w,x)$ generates $V\oplus\Lambda^3 V$ as a $\g_{ss}$-module, 
we obtain $P^2(V)=\operatorname{Coker}\mu=0$.

To check that $P^2(\Lambda^3 V)=0$, note that $(\Lambda^+ V\otimes(\Lambda^+ V\otimes \Lambda^3 V)^{sh})^{sh}=V\oplus\Lambda^3 V\oplus\Lambda^3 V$
and this projection is given by
$$\pi(v,w,y)=(\phi_3^1(\phi_3^3(x\otimes v)\otimes w),
\phi_1^3(\phi_3^1(y\otimes v)\otimes w),\phi_3^3(\phi_3^3(x\otimes v)\otimes w)),$$ 
$v,w\in\Lambda^+ V$, $y\in \Lambda^3 V$ and $\mu(v,w,y)=(alt\otimes 1)\circ \pi$, see Lemma \ref{technical2}. 
Choosing 
$$
v=e_1\wedge e_{-1}\wedge e_3 \wedge e_4+e_2\wedge e_{-2}\wedge e_3 \wedge e_4, \quad w=e_{-1}\wedge e_{-2}\wedge e_{-3}\wedge e_{-4}, \quad y=e_1\wedge e_2\wedge e_3
$$
we obtain that 
$$
\mu(v,w,y)=(2e_3,e_1\wedge e_{-1}\wedge e_3+e_2\wedge e_{-2}\wedge e_3, -2e_3\wedge e_4\wedge e_{-4}-e_1\wedge e_{-1}\wedge e_3-e_2\wedge e_{-2}\wedge e_3).
$$ 
Observe that $\g_{ss}$-submodule generated by $\mu(v,w,y)$ coincides with $V\oplus\Lambda^3 V\oplus\Lambda^3 V$. Hence 
$P^2(\Lambda^3 V)=\operatorname{Coker} \mu=0$.
\end{proof}

\begin{Cor}\label{dualityofp} If $\g=\so_n\niplus V$ or $\g=\so_8\niplus \Lambda^+V$ and $P^2(L)\neq 0$, 
then $Jor(P^2(L))$ is simple and 
coincides with the socle of 
$Jor(P(L)/rad^3P(L))$.
\end{Cor}
\begin{proof} Follows from direct description of $P(L)$.
\end{proof}

\begin{thm}\label{classification-unital}
\begin{enumerate}
\item  If $\g=\so_{2m+1}\niplus V$, then $A(J)={\bf k}({\bf Q}_1^{2m+1})'/I$, where $I$ is generated by the following relations with $r=2,\dots,m-1$
\begin{equation}\label{relations-standard-odd}\begin{array}{c}
\gamma_{r-1}\gamma_r=\delta_{r}\delta_{r-1}=0,\\ \gamma_{r-1} \delta_{r-1}=\delta_{r} \gamma_{r},\\ 
 \gamma_{m-1}\delta_{m-1}=\gamma_m^2.
\end{array}
\end{equation}
\item If $\g=\so_{2m}\niplus V$, then $A(J)={\bf k}({\bf Q}_1^{2m})'/I$, where $I$ is generated by the following relations with $r=2,\dots,m-2$
\begin{equation}\label{relations-standard-even}\begin{array}{c}
\gamma^{\pm}\delta^{\pm}=\gamma_{r-1}\gamma_r=\delta_{r}\delta_{r-1}=0,\\
\gamma_{r-1} \delta_{r-1}=\delta_{r} \gamma_{r},\\
\gamma_{m-2}\delta_{m-2}=\delta^+\gamma^+=\delta^-\gamma^-.
\end{array}
\end{equation}
\item If $\g=\so_8\niplus \Lambda^+ V$, then $A(J)={\bf k}{\bf Q}'_2/I$, where $I$ is generated by 
\begin{equation}\label{relations-lambda-plus}\begin{array}{c}
\alpha_1\beta_1=\beta_1\alpha_1=\beta_1\gamma_1=\gamma_1\alpha_1=\gamma^2_1=\alpha_2\gamma_2=\gamma_2\beta_2=\gamma_3\alpha_2=\beta_2\gamma_3=0,\\ 
\gamma_2^2=\beta_2\alpha_2,\quad \gamma^2_3=\alpha_2\beta_2.
\end{array}
\end{equation}
\item All above algebras are quadratic, satisfy $rad^3A(J)=0$. Furthermore, in the first two cases $A(J)$ is a Frobenius algebra.
\end{enumerate}
\end{thm}
\begin{proof} Corollary \ref{dualityofp} implies that all paths in $Q'(\g)$ of length $2$ leading from vertex $i$ to vertex $j$ are proportional 
with non-zero coefficients.
Moreover, after suitable  normalization one can make them equal.  

It is straightforward that the quadratic relations imply $rad^3A(J)=0$. Finally, in the first two cases $A(J)$ is a
Frobenius algebra since $P(L)^*\simeq P(L)$ if $L\neq \Lambda^\pm V$ and $m$ is odd. In the latter case
$P(\Lambda^\pm V)^*\simeq P(\Lambda^\mp V)$. 
\end{proof}

\subsection {Mixed case} Now we will deal with the case when $\g$ is (4),(5) or (6) from {\bf List A}. We will prove first some statements about more general 
situation.
Assume that $\g_{ss}=\g_l\oplus \g_r$, where $\g_l$ and $\g_r$ are simple Lie algebras and $R=\Gamma_l\boxtimes \Gamma_r$, where
$\Gamma_l\in \g_l-\text{mod}_{\frac{1}{2}}$ and $\Gamma_r\in \g_r-\text{mod}_{\frac{1}{2}}$. In our situation $\g_l$ and $\g_r$ are orthogonal Lie algebras,
hence both $\Gamma_l$ and $\Gamma_r$ are spinor modules. Since spinor modules are minuscule, then its tensor product with any irreducible module is multiplicity 
free. In particular, $S^2\Gamma_l$, $\Lambda^2\Gamma_l$ (respectively, $S^2\Gamma_r$, $\Lambda^2\Gamma_r$) are multiplicity free disjoint 
$\g^l$ (respectively, $\g^r$)-modules. 

Note that the $sl_2$-subalgebra $\alpha$ is the diagonal subalgebra in $\alpha_l\oplus\alpha_r$, where $\alpha_l$ and  $\alpha_{r}$ are $sl_2$-subalgebras in
$\g_l$ and $\g_r$ respectively. Therefore any module $M$ in $\ggm$ is equipped with
with $\mathbb Z/2 \oplus \mathbb Z/2$-grading 
\begin{equation}\label{doublegrading}
M=\bigoplus M_{(1,0)}\oplus M_{(0,1)}\oplus M_{(\frac 12,\frac 12)}\oplus M_{(0,0)}\oplus M_{(-1,0)}\oplus M_{(0,-1)}\oplus M_{(-\frac 12,-\frac 12)},
\end{equation}
such that short grading of $M$ with respect to $\alpha$ is given by 
$$M_k=\bigoplus_{i+j=k}M_{i,j}.$$

\begin{lem}\label{highestvector} Let $v$ be a highest weight vector in $L$.
\begin{enumerate}
\item $S^kR(v)$ generates $P^k(L)$.
\item If $w\in P^k(L)$ is a highest weight vector, then  $w\in S^kR(v)$. 
\end{enumerate}
\end{lem}
\begin{proof} Both assertions are obvious.
\end{proof}

\begin{lem}\label{radcubelem} 
(a) Let $L$ be a simple non-trivial module in $\ggm$. Then $L\in  \g_l-\text{mod}_{1}$, $L\in  \g_r-\text{mod}_{1}$ or $L$ is isomorphic 
to $A\boxtimes B$ for some simple $A\in \g_l-\text{mod}_{\frac{1}{2}}$ and $B\in \g_r-\text{mod}_{\frac{1}{2}}$.

(b) If $L\in  \g_l-\text{mod}_{1}$ (resp., $L\in  \g_r-\text{mod}_{1}$), then $P^3(L)=0$ and $P^2(L)\in  \g_r-\text{mod}_{1}$ 
(resp., $\g_l-\text{mod}_{1}$);

(c) If $L=A\boxtimes B$, then $P^3(L)$ is a trivial $\g_{ss}$-module.
\end{lem}
\begin{proof} (a) Follows easily from the double grading. Indeed, we have the following three possibilities
\begin{itemize}
\item $L=L_{1,0}\oplus L_{0,0}\oplus L_{-1,0}$;
\item $L=L_{0,1}\oplus L_{0,0}\oplus L_{0,-1}$;
\item $L=L_{\frac{1}{2},\frac{1}{2}}\oplus L_{-\frac{1}{2},-\frac{1}{2}}$.
\end{itemize}

(b) Without loss of generality assume that  $L=L_{1,0}\oplus L_{0,0}\oplus L_{-1,0}$.
Let $v\in L, w\in P^3(L)$ be $\g_{ss}$-highest weight vectors.   Then
$v\in L_{(1,0)}, w\in P^3(L)_{(\frac 12,\frac 12)}$ and by Lemma \ref{highestvector} we have 
$$w\in \sum_{i,j,k\in \{\pm\frac 12\}}R_{(-\frac 12, i)} R_{(-\frac 12, j)}  R_{(\frac 12, k)}v.$$
But  $R_{(\frac 12,k)}L_{(1,0)}=0$ by (\ref{doublegrading}). Contradiction.

If we assume that $w\in P^2(L)$ is a highest vector, then by the similar grading consideration we have
$$w\in \sum_{j,k,\in \{\pm\frac 12\}}R_{(-\frac 12, j)}  R_{(-\frac 12, k)}v.$$
This implies that the degree of $w$ is $(0,0)$ or $(0,1)$. Hence $P^2(L)\in  \g_r-\text{mod}_{1}$.

(c) Now let $L=L_{\frac{1}{2},\frac{1}{2}}\oplus L_{-\frac{1}{2},-\frac{1}{2}}$, $v\in L, w\in P^3(L)$ be $\g_{ss}$-highest weight vectors. 
We want to show that the degree of $w$ is $(0,0)$. Indeed, assume without loss of generality that degree of $w$ is $(1,0)$.
Then $v\in L_{\frac{1}{2},\frac{1}{2}}$ and we have 
$$w\in \sum_{i,j,k\in \{\pm\frac 12\}}R_{(-\frac 12,i)} R_{(\frac 12,j)}  R_{(\frac 12,k)}v.$$
But  $R_{(\frac 12,j)}  R_{(\frac 12,k)}L_{(\frac 12,\frac 12)}=0$ by (\ref{doublegrading}). Contradiction.
\end{proof}
\begin{Cor}\label{radcubecor} We have $rad^3 A(\J)=0$. 
\end{Cor}

\begin{lem}\label{mixedcase} 
Let $L$ be a simple non-trivial module in $\ggm$. Then the length of the radical filtration of the indecomposable projective $P(L)$ is at most $3$. 
Moreover, it is $3$ in one of the following cases
\begin{enumerate}
\item $L$ is a simple submodule in $S^2\Gamma_l^*$ (respectively, $\Lambda^2\Gamma_l^*$). Then $P^1(L)=\Gamma_l^*\boxtimes \Gamma_r$ and 
$P^2(L)=S^2 \Gamma_r$ (respectively, $\Lambda^2 \Gamma_r$).  
\item $L$ is a simple submodule in $S^2\Gamma_r^*$ (respectively, $\Lambda^2\Gamma_r^*$). Then $P^1(L)=\Gamma_l\boxtimes \Gamma_r^*$ and 
$P^2(L)=S^2 \Gamma_l$ (respectively, $\Lambda^2 \Gamma_l$). 
\item $L=\Gamma_l^*\boxtimes \Gamma_r^*$, then $P(L)$ is self-dual with $P^2(L)=\Gamma_r\boxtimes \Gamma_l$ and 
$P^1(L)=(\Gamma_l\otimes\Gamma_l^*\oplus \Gamma_r\otimes \Gamma_r^*)/tr$.
\end{enumerate}
\end{lem}
\begin{proof} We have to consider three cases as in Lemma \ref{radcubelem}. 

The first two cases are similar by symmetry. Therefore it is sufficient to consider the case $L\in\g_l-\text{mod}_1$.
Then we have
$$P^1(L)=(R\otimes L)^{sh}=(\Gamma_l\otimes L)^{sh}\boxtimes \Gamma_r.$$
Furthermore,
\begin{equation}\label{CD}
(R\otimes (R\otimes L)^{sh})^{sh}=C\oplus D,
\end{equation}
where $C\in  \g_l-\text{mod}_{1}$, $D \in  \g_r-\text{mod}_{1}$, and we assume in addition that $C^{\g_l}=0$.
Due to  Lemma \ref{radcubelem}(b) we know that $P^2(L)=\operatorname{Coker}\bar{\mu}$, where $\bar{\mu}$ is the composition of $\mu$, defined in
Lemma \ref{technical2}, with the natural projection on $D$.
More precisely, $D=(\Gamma_l\otimes (\Gamma_l\otimes L)^{sh})^{\g_l}\boxtimes (\Gamma_r\otimes\Gamma_r)$.
If $D\neq 0$, then $(\Gamma_l\otimes \Gamma_l\otimes L)^{\g_l}\neq 0$, which is only possible when $L$ is a simple submodule in $(\Gamma_l\otimes\Gamma_l)^*$.
In this case the multiplicity of $L$ in  $(\Gamma_l\otimes\Gamma_l)^*$ is one. Therefore $D=\Gamma_r\otimes\Gamma_r$.

Let $L$ be a submodule in $S^2\Gamma_l^*$. Let $v_1,v_2\in\Gamma_l$, $w_1,w_2\in\Gamma_r$ and $x \cdot y\in S^2\Gamma_l^*$. Then
$$\bar\mu (v_1\boxtimes v_2, w_1\boxtimes w_2, x\cdot y)=(\langle x,v_1\rangle \langle y,v_2 \rangle+\langle y,v_1\rangle \langle x,v_2 \rangle)w_1\otimes w_2- 
(\langle x,v_2\rangle \langle y,v_1 \rangle+\langle y,v_2\rangle \langle x,v_1 \rangle)w_2\otimes w_1.$$
Thus, $\operatorname{Im}\bar\mu=\Lambda^2\Gamma_r$ and hence $P^2(L)=S^2\Gamma_r$.

In the similar way, with the change of sign, one obtains that if $L$ is a submodule in $\Lambda^2\Gamma_l^*$, then 
$\operatorname{Im}\bar\mu=S^2\Gamma_r$ and hence $P^2(L)=\Lambda^2\Gamma_r$.

We also have an explicit construction of $P(L)$. Assume for example that $L\subset S^2\Gamma_l^*$. There exists an indecomposable
module $M$ of length two in $\gmh$ with submodule $\Gamma_r$ and quotient $\Gamma_l^*$. Then $S^2M\in\ggm$ is indecomposable, with the radical filtration:
$M^0=S^2\Gamma_l^*$, $M^1=\Gamma_l^*\boxtimes \Gamma_r$, $M^2=S^2\Gamma_r$. One can check that
$L\subset M_0$ generates a submodule isomorphic to $P(L)$. 

Now assume that $L=A\boxtimes B$ as in Lemma \ref{radcubelem}(c). 
If $A\neq \Gamma_l^*$ and $B\neq \Gamma_r^*$, then $P(L)^1=(L\otimes R)^{sh}=0$
and hence $P(L)=L$. Assume next that $A\neq \Gamma_l^*$ and $B=\Gamma_r^*$. Then
$$P^1(L)=\Gamma_l\otimes A$$
and 
$$(R\otimes (R\otimes L))^{sh}=(\Gamma_l\otimes\Gamma_l\otimes A)^{sh}\boxtimes \Gamma_r.$$
Furthermore, if we denote by $\pi_l$ the natural projection $\Gamma_l\otimes\Gamma_l\otimes A\to (\Gamma_l\otimes\Gamma_l\otimes A)^{sh}$,
 then
for all $v_1,v_2\in\Gamma_l$, $w_1,w_2\in\Gamma_r$ and $a\in A, b\in B$ we have 
$$\mu (v_1\boxtimes w_1, v_2\boxtimes w_2, a\boxtimes b)=
\langle b,w_2\rangle \pi_l(v_1,v_2,a)\boxtimes w_1-\langle b,w_1\rangle \pi_l(v_2,v_1,a)\boxtimes w_2.$$
Since $\pi_l$ is surjective, we obtain $\operatorname{Coker}\mu=0$ and hence $P^2(L)=0$.

Finally, we assume that $L=\Gamma_l^*\boxtimes \Gamma_r^*$. In this case
$$P^1(L)=(R\otimes L)^{sh}=tr \oplus (\Gamma_l\otimes\Gamma_l^*)_0\oplus (\Gamma_r\otimes \Gamma_r^*)_0,$$
where $(X\otimes X^*)_0$ denotes the traceless part of $X\otimes X^*$.
Then $$(R\otimes (R\otimes L)^{sh})^{sh}=\Gamma_l\boxtimes \Gamma_r\oplus (\Gamma_l\otimes(\Gamma_l\otimes\Gamma_l^*)_0)^{sh}\boxtimes \Gamma_r\oplus
\Gamma_l\boxtimes(\Gamma_r\otimes(\Gamma_r\otimes\Gamma_r^*)_0)^{sh}.$$

We notice that $(\Gamma_l\otimes(\Gamma_l\otimes\Gamma_l^*)_0)^{sh}\simeq \Gamma_l$ with the natural projection
$\pi_l:\Gamma_l\otimes(\Gamma_l\otimes\Gamma_l^*)_0\to \Gamma_l$ given by the formula
$$\pi_l(v_1, v_2, x)=\langle x,v_1\rangle v_2-\langle x,v_2\rangle v_1.$$
We have analogous formula for $\pi_r:\Gamma_r\otimes(\Gamma_r\otimes\Gamma_r^*)_0\to \Gamma_r$.
Then $\pi:R\otimes R\otimes L\to (R\otimes (R\otimes L)^{sh})^{sh}$ is defined by
$$\pi (v_1\boxtimes w_1, v_2\boxtimes w_2, x\boxtimes y)=(\langle x,v_2\rangle \langle y,w_2\rangle v_1\boxtimes w_1,
\langle y,w_2\rangle\pi_l(v_1,v_2,x)\boxtimes w_1,\langle x,v_2\rangle v_1\boxtimes\pi_r(w_1,w_2,y).$$
By tedious straightforward calculations one obtains that $\theta:(R\otimes (R\otimes L)^{sh})^{sh}\to \Gamma_l\boxtimes\Gamma_r$
defined by $\theta(x_1,x_2,x_3)=2x_1-x_2-x_3$ gives the cokernel of $\mu$.
Therefore, we obtain $P^2(L)=  \Gamma_l\boxtimes\Gamma_r$.
Finally, let us prove that $P^3(L)=0$. Assume the opposite. Then $P^3(L)$ is trivial and every submodule $N$ generated by a 
simple submodule $L'$
in $P^1(L)$ has Loewy length $3$ with trivial submodule in $rad^3 N$. That contradicts description of $P^1(L)$ and $P^2(L')$.

We leave to the reader to check that $P(L)\simeq M\otimes M^*$ where $M$ is defined above in this proof.   
\end{proof}

We use the last lemma in order to determine $A(J)$, equivalently the relations in $Q'(\g)$, when $\g$ in {\bf List A} 
and $\g_{ss}=\so_{2m}\oplus\so_{2n}$, $m,n\in\{4,5\}$. 
\begin{thm}\label{quivers-mixed}
\begin{enumerate}
\item If $\g=(\so_8\oplus\so_8)\niplus\Gamma^+_1\boxtimes\Gamma^+_2$, then $A(\J)={\bf k}{\bf Q}'_3/I$, where the ideal $I$ is generated by
\begin{equation}\label{relations-so8}
\begin{array}{c}
\alpha_i\beta_i=\alpha_j\beta_j,\  \beta_i\alpha_i=\delta_i\tau_j=\tau_i\delta_j=0,\ 1\leq i,j\leq 4;\\ 
\beta_2\alpha_1=\beta_4\alpha_1=\beta_1\alpha_2=\beta_3\alpha_2=\beta_2\alpha_3=\beta_4\alpha_3=\beta_1\alpha_4=\beta_3\alpha_4=0.
\end{array}
\end{equation}

\item  If $\g=(\so_8\oplus\so_{10})\niplus\Gamma^+_1\boxtimes\Gamma^+_2$, then $A(\J)={\bf k}{\bf Q}'_4/I$, where the ideal $I$ is 
generated by
\begin{equation}\label{relations-so8-so10}
\begin{array}{c}
\beta_i\alpha_i=\beta_j\alpha_j,\ 1\leq i,j\leq 4;\ \alpha_i\delta_1=\alpha_i\delta_3=0,\ i\neq 4;\ \alpha_i\delta_2=0,\ i\neq 3;\ 
\gamma_i\beta_3=0,\ i=1,3;\\ 
\gamma_i\beta_4=0,\ i=1,2;\ \gamma_i\beta_1=\gamma_i\beta_2=0,\ 1\leq i\leq 3;\ \delta_j\gamma_j=0,\ j=1,2;\ \tau_i\rho_i=0,\ i=1,2.
\end{array}
\end{equation}

\item If $\g=(\so_{10}\oplus\so_{10})\niplus\Gamma^+_1\boxtimes\Gamma^+_2$, then  $A(\J)={\bf k}{\bf Q}'_5/I$, where 
the ideal $I$ is generated by
\begin{equation}\label{relations-so10}
\begin{array}{c}
\alpha_1\beta_2=\alpha_3\beta_2=\alpha_2\beta_1=\alpha_2\beta_3=\gamma_1\delta_2=\gamma_3\delta_2=\gamma_2\delta_1=\gamma_2\delta_3=0;\ 
\rho_i\tau_i=\rho_j\tau_j,\ 1\leq i,j\leq 4.
\end{array}
\end{equation}
\end{enumerate}
\end{thm}
\begin{proof}
In order to write down the relations in the path algebra ${\bf k}Q'(\g)$ it is enough to describe all projective covers
 of simple non-trivial modules of Loewy length three. (Recall that by Lemma\,\ref{radcubelem} all indecomposable projectives have Loewy length at most three).

Let $\g=(\so_8\oplus\so_8)\niplus\Gamma^+_1\boxtimes\Gamma^+_2$. For $\so_8$-spinor module $\Gamma^+$ we have
\begin{equation}\label{gamma-so8}
(\Gamma^+)^*=\Gamma^+,\quad S^2(\Gamma^+)=\Lambda^+ V\oplus tr,\quad \Lambda^2(\Gamma^+)=\Lambda^2 V.
\end{equation}
 By Lemma\,\ref{mixedcase} we obtain the following indecomposable projectives of Loewy length three.
\begin{equation}\label{projective-so8-so8}
\begin{array}{c}
\Lambda^+ V\\
\overline{\Gamma_1^+\boxtimes\Gamma_2^+ }\\
\overline{\Lambda^+ W \oplus tr}\\
\end{array} \quad
\begin{array}{c}
\Lambda^+ W\\
\overline{\Gamma_1^+\boxtimes\Gamma_2^+ }\\
\overline{\Lambda^+ V \oplus tr}\\
\end{array} \quad
\begin{array}{c}
\Lambda^2 V\\
\overline{\Gamma_1^+\boxtimes\Gamma_2^+ }\\
\overline{\quad \Lambda^2 W \quad}\\
\end{array} \quad
\begin{array}{c}
\Lambda^2 W\\
\overline{\Gamma_1^+\boxtimes\Gamma_2^+ }\\
\overline{\quad \Lambda^2 V \quad}\\
\end{array}\quad
\begin{array}{c}
\Gamma_1^+\boxtimes\Gamma_2^+\\
\overline{tr\oplus\Lambda^2 V\oplus\Lambda^+ V\oplus\Lambda^2 W\oplus\Lambda^+ W} \\
\overline{\qquad \qquad \Gamma_1^+\boxtimes\Gamma_2^+\qquad\qquad }\\
\end{array} 
\end{equation}
The relations in $A(\J)=\kk Q'_3/I$ follow from \eqref{projective-so8-so8}. They imply $rad^3 A(\J)=0$.

Let $\g=(\so_8\oplus\so_{10})\niplus\Gamma^+_1\boxtimes\Gamma^+_2$. For $\so_{10}$-spinor modules $\Gamma_2^\pm$ we have
\begin{equation}\label{gamma-so10}
(\Gamma_2^+)^*=\Gamma_2^-,\quad S^2(\Gamma_2^\pm)=\Lambda^\pm W\oplus W,\quad \Lambda^2(\Gamma_2^\pm)=\Lambda^3 W, \quad
\Gamma_2^+\otimes \Gamma_2^-=tr\oplus \Lambda^2W\oplus\Lambda^4W.
\end{equation}

By Lemma\,\ref{mixedcase} the indecomposable projective modules of Loewy length three in $\ggm$ are the following:
\begin{equation}\label{projective-so8-so10}
\begin{array}{c}
\Lambda^2 V\\
\overline{\Gamma_1^+\boxtimes\Gamma_2^+ }\\
\overline{\quad \Lambda^3 W \quad}\\
\end{array} \quad
\begin{array}{c}
\Lambda^+ V\\
\overline{\Gamma_1^+\boxtimes\Gamma_2^+ }\\
\overline{\Lambda^+ W \oplus W}\\
\end{array} \quad
\begin{array}{c}
\Lambda^3 W\\
\overline{\Gamma_1^+\boxtimes\Gamma_2^- }\\
\overline{\quad \Lambda^2 V \quad}\\
\end{array} \quad
\begin{array}{c}
\Lambda^- W\\
\overline{\Gamma_1^+\boxtimes\Gamma_2^- }\\
\overline{\Lambda^+ V \oplus tr}\\
\end{array} \quad
\begin{array}{c}
W\\
\overline{\Gamma_1^+\boxtimes\Gamma_2^- }\\
\overline{\Lambda^+ V \oplus tr}\\
\end{array}\quad
\begin{array}{c}
\Gamma_1^+\boxtimes\Gamma_2^-\\
\overline{tr\oplus\Lambda^2 V\oplus \Lambda^+ V\oplus\Lambda^2 W\oplus\Lambda^4 W} \\
\overline{\qquad \qquad \Gamma_1^+\boxtimes\Gamma_2^+\qquad\qquad }\\
\end{array} 
\end{equation}
The relations \eqref{relations-so8-so10} in $A(\J)$ follow and imply $rad^3A(\J)=0$.

Finally, if $\g=(\so_{10}\oplus\so_{10})\niplus\Gamma^+_1\boxtimes\Gamma^+_2$, using \eqref{gamma-so10}, we apply 
Lemma\,\ref{mixedcase} to obtain the indecomposable projectives in $\ggm$:
\begin{equation}\label{projective-so10-so10}
\begin{array}{c}
\begin{array}{c}
\Lambda^- V\\
\overline{\Gamma_1^-\boxtimes\Gamma_2^+}\\
\overline{ \Lambda^- W\oplus W }\\
\end{array} \quad
\begin{array}{c}
V\\
\overline{\Gamma_1^-\boxtimes\Gamma_2^+}\\
\overline{\Lambda^- W \oplus W }\\
\end{array} \quad
\begin{array}{c}
\Lambda^3 V\\
\overline{\Gamma_1^-\boxtimes\Gamma_2^+ }\\
\overline{\Lambda^3 W }\\
\end{array} \quad
\begin{array}{c}
\Lambda^- W\\
\overline{\Gamma_1^+\boxtimes\Gamma_2^- }\\
\overline{\Lambda^- V\oplus V }\\
\end{array} \quad
\begin{array}{c}
W\\
\overline{\Gamma_1^+\boxtimes\Gamma_2^- }\\
\overline{\Lambda^- V\oplus V }\\
\end{array}\quad
\begin{array}{c}
\Lambda^3 W\\
\overline{\Gamma_1^+\boxtimes\Gamma_2^- }\\
\overline{\Lambda^3 V }\\
\end{array}\\
\begin{array}{c}
\Gamma_1^-\boxtimes\Gamma_2^-\\
\overline{tr\oplus\Lambda^2 V\oplus\Lambda^4 V\oplus\Lambda^2 W\oplus\Lambda^4 W} \\
\overline{\qquad \qquad \Gamma_1^+\boxtimes\Gamma_2^+\qquad\qquad }\\
\end{array} 
\end{array}
\end{equation}
The relations \eqref{relations-so10} in $A(\J)$ follow and imply $rad^3A(\J)=0$.
\end{proof}

\subsection{Tameness}
\begin{thm} The algebras $A(J)$ described in Theorem\,\ref{classification-unital} and Theorem\,\ref{quivers-mixed} are tame.
\end{thm}
\begin{proof}
First, we deal with the cases $\g=\so_{n}\niplus V$, $\g=\so_8\niplus\Lambda^+ V$ and $\g=(\so_8\oplus\so_8)\niplus\Gamma_1^+\boxtimes\Gamma_2^+$. 
We note that $A(J)$ 
satisfies the conditions of Theorem\,\ref{Viktor}, Theorem\,\ref{Gabriel}\eqref{double-quiver} implies that $Jor(\so_{n}\niplus V)$ is
 of finite type, 
while $Jor(\so_8\niplus\Lambda^+)$ and $Jor((\so_8\oplus\so_8)\niplus\Gamma_1^+\boxtimes\Gamma_2^+)$ are tame.

Next, let $\g=(\so_8\oplus\so_{10})\niplus\Gamma^+_1\boxtimes\Gamma^+_2$ then $A(\g)=A^1\oplus A^2\oplus A^3$ is the direct sum 
of subalgebras such that
$A^1={\bf k}$, $Q(A^2)$ has four vertices and $Q(A^3)$ has ten vertices, see  \eqref{so_8-so_10}. 
Observe that $rad^2 A^2=0$ therefore $A^2$ is of finite representation type by Theorem\,\ref{Gabriel}\eqref{double-quiver}. 

Let us show that $A^3$ is tame. Let $I_1=\langle \alpha_i\beta_i\rangle_{1\leq i\leq 4}=rad^2 A^3 e_{\Gamma_1^+\boxtimes\Gamma_2^-}$, put $B=A^3/ I_1$.
Note that the projective $P(\Gamma_1^+\boxtimes\Gamma_2^-)$ satisfies the condition of Lemma\,\ref{Viktor1}. 
Therefore an indecomposable $A^3$-module non-isomorphic $P(\Gamma_1^+\boxtimes\Gamma_2^-)$ is, in fact, a
module over the algebra $B$.

Furthermore divide the set of $V(Q(A^3))$ into three subsets, namely $S_l=\{\Lambda^- W, \Gamma^+_1\boxtimes\Gamma^-_2\}$, 
$S_r=\{\Lambda^+ W, \Gamma_1^+\boxtimes\Gamma^+_2\}$ and $T=\{W, \Lambda^2 W, \Lambda^3 W,\Lambda^4 W, \Lambda^2 V, \Lambda^+ V\}$ 
then by Lemma\,\ref{divide} any indecomposable $B$-module $M$ is either 
a $B'=e_{S_l}B e_{S_l}$-module or  a $B''=e_{S_r}B e_{S_r}$-module. The Ext quivers of $B'$ and $B''$ are the following:
\begin{equation}
\xymatrix{
W&&\Lambda^2 W \ar[ddl]_{\beta_1}  &&& \Lambda^2 W &&W\ar[ddl]_{\delta_1}\\
\Lambda^3 W&& \Lambda^4 W \ar[dl]_{\beta_2}  &&& \Lambda^4 W &&  \Lambda^3 W \ar[dl]_{\delta_2}\\
\Lambda^- W & {\Gamma_1^+\boxtimes\Gamma_2^-}\ar[l]^{\gamma_3}\ar[uul]_{\gamma_1}\ar[ul]_{\gamma_2} & \Lambda^2 V \ar[l]_{\beta_3}&&&   \Lambda^2 V 
& {\Gamma_1^+\boxtimes\Gamma_2^+}\ar[ul]_{\alpha_2}\ar[l]_{\alpha_3}\ar[dl]^{\alpha_4} \ar[uul]_{\alpha_1} & \Lambda^+ W \ar[l]_{\delta_3}\\
&& \Lambda^+ V \ar[ul]^{\beta_4} &&& \Lambda^+ V &&\\}
\end{equation}
Note that $(B')^{op}$-mod is equivalent to $B''$-mod thus it is sufficient to determine the type of $B'$. Since $Q(B')$
is a tree we determine its representation type by calculating the Tits form, see \eqref{Tits-form} in Appendix. It can be written in the following form
$$
q_{B'}(x)=(x_1+x_5+x_6-x_4)^2+(x_2+x_3+x_7+x_8-x_4)^2+\sum_{i=2,5,7}(x_i-x_{i+1})^2+x_1^2+2x_6(x_2+x_3).
$$
One can see that $q_{B'}$ is weakly non-negative, therefore $B'$ as well as $B''$ are of tame representation type.
That implies that $A^3$ is tame.

The last case is $\g=(\so_{10}\oplus\so_{10})\niplus\Gamma^+_1\boxtimes\Gamma^+_2$, and $A(\g)=A^1\oplus A^2$, where
$Q(A^1)$ has seven vertices while $Q(A^2)$ has ten vertices, see  \eqref{so_10-so_10}. 
By Theorem\,\ref{Viktor} the algebra $A^1$ is tame. To prove tameness of $A_2$ we split $V(Q(A^2))$ into 
into three subsets, namely $S_l=\{\Lambda^- W, \Lambda^+ V, \Gamma^+_1\boxtimes\Gamma^-_2\}$, 
$S_r=\{\Lambda^+ W, \Lambda^- V, \Gamma_1^-\boxtimes\Gamma^+_2\}$ and $T=\{W, \Lambda^3 W, \Lambda^3 V, V\}$ 
then by Lemma\,\ref{divide} any indecomposable $A^2$-module is either 
$C=e_{S_l}A^2 e_{S_l}$-module or $C'=e_{S_r}A^2 e_{S_r}$-module, where $Q(C)$ and $Q(C')$ are respectively
 \begin{equation}
\xymatrix{
W&& V \ar[ddl]_{\beta_1}  &&& V &&W\ar[ddl]_{\delta_1}\\
\Lambda^3 W&& \Lambda^3 V \ar[dl]_{\quad\beta_2}  &&& \Lambda^3 V &&  \Lambda^3 W \ar[dl]_{\quad\delta_2}\\
\Lambda^- W & {\Gamma_1^+\boxtimes\Gamma_2^-}\ar[l]^{\alpha_3}\ar[uul]_{\alpha_1}\ar[ul]_{\alpha_2} & \Lambda^+ V \ar[l]_{\quad\beta_3}&&&   \Lambda^- V 
& {\Gamma_1^-\boxtimes\Gamma_2^+}\ar[ul]_{\gamma_2}\ar[l]_{\gamma_3}\ar[uul]_{\gamma_1} & \Lambda^+ W \ar[l]_{\quad\delta_3}\\}
\end{equation}
Like in the previous case $C^{op}$-mod is equivalent to $C'$-mod therefore it is sufficient to determine the type of $C$. The Tits form corresponding to $C$
\begin{equation}
q_C(x)=(x_1+x_3+x_5+x_7-x_4)^2+(x_2+x_6-x_4)^2+(x_1-x_3)^2+(x_5-x_7)^2+x_2^2+x_6^2
\end{equation}
is weakly non-negative. Therefore, by Theorem\,\ref{Bruestle}, $C$ is tame. This finishes the proof.
\end{proof}

\section{Algebras with $[\rr,\rr]\neq 0$}
In view of Theorem \ref{main-quiver} our next step is to consider indecomposable Lie algebras $\g$ with short grading, 
irreducible $R=\rr/[\rr,\rr]$ such that $Q(\g)$ is admissible and
$[\rr,[\rr,\rr]]=0$. 
Irreducibility of $R$ implies that $\rr$ is a 
nilpotent Lie algebra of nilindex two, and $[\rr,\rr]$ is a $\g_{ss}$-submodule in $(\Lambda^2R)^{sh}$. 
Therefore we have to consider the following cases
\begin{enumerate}
\item $\g_{ss}=\so_m$, $R=V$, $R^2=(\Lambda^2 V)^{sh}=\Lambda^2 V$;
\item $\g_{ss}=\so_8$, $R=\Lambda^+$, $R^2=(\Lambda^2 R)^{sh}=\Lambda^2V$;
\item $\g_{ss}=\so_8\oplus \so_8$, $R=\Gamma^+_1\boxtimes \Gamma^+_2$. In this case $[\rr,\rr]=(\Lambda^2 R)^{sh}=\Lambda^2V\oplus \Lambda^2W$, hence
$[\rr,\rr]=(\Lambda^2 R)^{sh}$ or $[\rr,\rr]=\Lambda^2V$ or  $[\rr,\rr]=\Lambda^2W$, and the last two cases are clearly isomorphic;
\item $\g_{ss}=\so_8\oplus \so_{10}$, $R=\Gamma^+_1\boxtimes \Gamma_2$. In this case $(\Lambda^2 R)^{sh}=\Lambda^3 W$ is irreducible;
\item $\g_{ss}=\so_{10}\oplus \so_{10}$, $R=\Gamma_1\boxtimes \Gamma_2$. In this case $(\Lambda^2 R)^{sh}=0$, therefore it does not need further consideration.
\end{enumerate} 
In what follows we refer to the above list as {\bf List B}.
In this section we will prove that $A(J)$ is wild for the algebras (1)-(4) in {\bf List B} .

\begin{lem}\label{standard2} Let  $\g_{ss}=\so_m$, $R=V$, $[R,R]=\Lambda^2 V=(\Lambda^2 V)^{sh}$. Then $A(\g)=\kk {\bf Q}^1_m/I_2$, where 
$I_2\subset rad^3\kk {\bf Q}^1_m$, i.e all the relations are of degree $3$ or higher.
\end{lem}
\begin{proof} The proof amounts to showing that for a simple $L\in\ggm$, we have $$P^2(L)=(R\otimes (R\otimes L)^{sh})^{sh}.$$ We use Lemma \ref{technical2}.
Since $\delta:\Lambda^2R\to R^2$ is an isomorphism, the map 
$$\mu:\Lambda^2R\otimes L\to (R\otimes (R\otimes L)^{sh})^{sh}$$ is the composition
$$\Lambda^2R\otimes L\xrightarrow{\lambda}(R^2\otimes L)^{sh}\xrightarrow{alt\otimes 1} (R\otimes R\otimes L)^{sh}\rightarrow (R\otimes (R\otimes L)^{sh})^{sh}.$$
Hence $P^2(L)=\operatorname{Coker} (\mu\oplus\lambda)=(R\otimes (R\otimes L)^{sh})^{sh}$. 
\end{proof}

\begin{lem}\label{jstandard2} Let $\g$ be as above and $\J=Jor(\g)$. Then the Ext quiver of $\JJ$ is

\begin{equation}\label{qui-relations-odd} m {\text\ is \ odd}: \quad \xymatrix{ 
V\ar@/^0.4pc/[r]^{\gamma_1}\ar_{\bar\gamma_0}@(ul,dl) & \Lambda^2 V \ar@/^0.4pc/[r]
^{\gamma_2} \ar@/^0.4pc/[l]^{\delta_1}&\dots 
\ar@/^0.4pc/[l]^{\delta_2}\ar@/^0.4pc/[r] & \Lambda^{m-1} 
V\ar@/^0.4pc/[r]^{\gamma_{m-1}}\ar@/^0.4pc/[l]& 
\Lambda^m V\ar@/^0.4pc/[l]^{\delta_{m-1}}\ar@{->}^{\gamma_{m}}@(ur,dr)}
\end{equation}

\begin{equation}\label{qui-relations-even}
\xymatrix{ &&&&& \Lambda^+\ar@/^0.4pc/[dl]^{\delta^+}\\
m {\text \ is \  even}: & V\ar_{\bar\gamma_0}@(ul,dl)\ar@/^0.4pc/[r]^{\gamma_1}& 
\Lambda^2 V\ar@/^0.4pc/[r]\ar@/^0.4pc/[l]^{\delta_1} & \dots \ar@/^0.4pc/[r]^{\gamma_{m-2}\ }\ar@/^0.4pc/[l] &
\Lambda^{m-1} V\ar@/^0.4pc/[l]^{\delta_{m-2}\ }\ar@/^0.4pc/[ur]^{\gamma^+}\ar@/^0.4pc/[dr]^{\gamma^-}&\\
&&&&&\Lambda^-\ar@/^0.4pc/[ul]^{\delta^-}}
\end{equation}
and the relations in $A(\J)$ lie in $rad^3 A(\J)+\bar\gamma_0 rad A(\J)$. 
\end{lem}
\begin{proof} It is a straightforward consequence of Lemma \ref{idempotent}. The appearance of new arrow $\bar\gamma_0$ follows 
from Lemma \ref{standard2}.
Indeed, recall 
${\bf Q}_1^{2m+1}$ and ${\bf Q}_1^{2m}$ of Theorem \ref{main-quiver}. Now $\bar\gamma_0:=\gamma_0\delta_0$ lies in the radical of 
$A(\J)$ since $\gamma_0\delta_0$ and
$\delta_1\gamma_1$ are linearly independent. 
\end{proof}

\begin{lem}\label{so8unital2} Let $\g=\so_8$, $R=\Lambda^+ V$ and $R^2=\Lambda^2 V$. 
Then the block of $\JJ$ containing $\Lambda^2V$ and 
$\Lambda^+ V$ has the quiver
\begin{equation}\label{quiver-lambda-plus1} \xymatrix 
{\Lambda^2 V\ar@/^0.4pc/[r]^{\alpha_2}\ar@{<-}^{\gamma_2}@(dl,ul) & \Lambda^+ V\ar@{<-}_{\gamma_3}@(dr,ur)\ar@/^0.4pc/[l]^{\beta_2}}
\end{equation}
and the only relations in this block modulo $rad^3$ are $$\alpha_2\gamma_2=\gamma_3\alpha_2,\quad \gamma_2\beta_2=\beta_2\gamma_3.$$ 
\end{lem}
\begin{proof} In order to prove this lemma, we have to compute
$P^2(\Lambda^2 V)$ and $P^2(\Lambda^+ V)$. Using symmetry of the Dynkin diagram of $\so_8$ we identify $\Lambda^+ V$ with $S^2V_0$ 
(the latter stands for the traceless part of $S^2V$)  as
in the proof of Proposition \ref{projectives-so-lambda}. 
If we identify $\so_8$ with the space of skew-symmetric matrices and $R=S^2V_0$ with the space of traceless symmetric matrices, 
the map $\delta:S^2 V_0\times S^2 V_0\to \Lambda^2 V$ is given by the usual commutator of matrices. 

We will do this computation for the case $L=S^2V_0$ leaving the second case to the reader.
We have 
$$(R\otimes (R\otimes L)^{sh})^{sh}=\Lambda^2V\oplus\Lambda^2V\oplus S^2V\oplus S^2V_0,\quad (R^2\otimes L)^{sh}=\Lambda^2V\oplus S^2V_0,$$
and
$$\pi(X\otimes Y\otimes A)=\left(\{X,[Y,A]\},[X,\{Y,A\}],\{X,\{Y,A\}\},[X,[Y,A]]\right),$$
where $\{C,B\}=CB+BC$ and $[C,B]=CB-BC$.

The canonical projection $p:R^2\otimes L\to (R^2\otimes L)^{sh}$ is given by
$$p(Z\otimes A)=\left(\{Z,A\},[Z,A]\right),$$
where $A\in S^2V_0, Z\in\Lambda^2V$. Therefore
$$\lambda(X,Y,A)=-\left(\{[X,Y],A\},[[X,Y],A]\right).$$

Define the map
$$\theta:(R\otimes (R\otimes L)^{sh})^{sh}\oplus(R^2\otimes L)^{sh}=
\Lambda^2V\oplus\Lambda^2V\oplus S^2V\oplus S^2V_0\oplus \Lambda^2V\oplus S^2V_0\to \Lambda^2V\oplus S^2V\oplus S^2V_0$$ 
by $\theta(x_1,x_2,x_3,x_4,x_5,x_6)=(x_1+x_2-2x_5,x_3-x_4,x_4-x_6)$. We claim that the sequence
$$\Lambda^2 R\otimes L\xrightarrow{\mu\oplus\lambda} (R\otimes (R\otimes L)^{sh})^{sh}\oplus(R^2\otimes L)^{sh}\xrightarrow{\theta} 
\Lambda^2V\oplus S^2V\oplus S^2V_0\to 0$$
is exact. Surjectivity of $\theta$ is trivial. The identity $\theta(\lambda\oplus\mu)=0$ follows from the following identities
\begin{equation}
\begin{array}{l}
{[X,\{Y,A\}]+[Y,\{X,A\}]=2\{XY,A\}},\\
{\{X,\{Y,A\}\}-[X,[Y,A]]=2XAY+2YAX},\\ 
{[X,[Y,A]]-[Y,[X,A]]=[[X,Y],A].}\\
\end{array}
\end{equation}
Balancing the numbers of $\g_{ss}$ components implies that $\operatorname{Ker}\theta=\operatorname{Im}(\lambda\oplus\mu)$.
Hence we obtain
$$P^2(S^2V_0)\simeq \Lambda^2V\oplus S^2V\oplus S^2V_0.$$
Similarly,
$$P^2(\Lambda^2V)=\Lambda^2V\oplus\Lambda^2V\oplus S^2V.$$
The relations modulo $rad^3$ follow by Lemma \ref{idempotent}. 
\end{proof}
\begin{lem}\label{mixedso8} 
Let $\g_{ss}=\so_8\oplus \so_8$ or $\so_8\oplus \so_{10}$, $R=\Gamma_1^+\boxtimes \Gamma_2^+$, $R^2=\Lambda^2 \Gamma_2^+$ and
the commutator $R\otimes R\to R^2$ be defined by the formula
$$[v_1\boxtimes w_1,v_2\boxtimes w_2]=(v_1,v_2)w_1\wedge w_2,$$
where $v_1,v_2\in \Gamma_1^+$,  $w_1,w_2\in \Gamma_2^+$ and $(\cdot,\cdot)$ is an invariant scalar product in $\Gamma_1^+$.
Consider the subcategory of $\ggm$ generated by simple submodules in $(\Gamma_2^+\otimes \Gamma_2^+)^*$ and 
$\Gamma_1^+\boxtimes (\Gamma_2^+)^*$.
Let $\bar P(L)$ denote the projective cover of $L$ in this subcategory. Then for any simple submodule $L$ of 
$(\Gamma_2^+\otimes \Gamma_2^+)^*$ we have
$\bar P^2(L)=\Gamma_2^+\otimes (\Gamma_2^+)^*$.
\end{lem}
\begin{proof} It is clear from the quiver $Q(\g)$ that $\bar P^2(L)$ is a submodule of $\Gamma_2^+\otimes (\Gamma_2^+)^*$. 

We claim that there exists an indecomposable module $M$ in $\gmh$ with the radical filtration 
$$M^0=(\Gamma_2^+)^*,\quad  M^1=\Gamma^+_1,\quad  M^2=\Gamma_2^+.$$ 
To show it, we define the action of $R$ and $R^2$ on $M$ by the formulas:
$$
\begin{array}{c}
v\boxtimes w(x_0,x_1,x_2)=(0,\langle w,x_0\rangle v,(v,x_1)w) \\
w_1\wedge w_2 (x_0,x_1,x_2)=(0,0, \langle w_1,x_0\rangle w_2-\langle w_2,x_0\rangle w_1)
\end{array}
$$
for all $v\in\Gamma_1^+$, $w_1$, $w_2\in\Gamma_2^+$, $x_i\in M^i$, $i=0,1,2$. Then we have
\begin{equation}
\begin{array}{c}
v_1\boxtimes w_1(v_2\boxtimes w_2)(x_0,x_1,x_2)=(0,0,\langle w_2,x_0\rangle(v_1,v_2) w_1)\\
v_2\boxtimes w_2(v_1\boxtimes w_1)(x_0,x_1,x_2)=(0,0,\langle w_1,x_0\rangle(v_2,v_1) w_2)
\end{array}
\end{equation}
and the reader can see that 
$$
[v_1\boxtimes w_1, v_2\boxtimes w_2]=(v_1,v_2) w_1\wedge w_2.
$$
Thus $M$ is a $\g$-module, it is indecomposable by construction.

Let $T=M\otimes (\Gamma_2^+)^*$. Then $T$ is indecomposable with 
radical filtration $$T^0=(\Gamma_2^+\otimes \Gamma_2^+)^*,\quad  T^1=\Gamma^+_1\boxtimes\Gamma_2^*,
\quad T^2= \Gamma_2^+\otimes (\Gamma_2^+)^*.$$ Since $T^1$ is simple,
any submodule $L$ of $T^0$ generates an indecomposable submodule $S$ with radical filtration $S^0=L, S^1=T^1, S^2=T^2$. 
But then $S$ is a quotient of 
$\bar P(L)$ and since $\bar P^2(L)\subset T^2=S^2$ we have $S\simeq \bar P(L)$. 
\end{proof}

\begin{Cor}\label{mixed2}

(a) Let $\g_{ss}=\so_8\oplus \so_8$, $R=\Gamma^+_1\boxtimes \Gamma^+_2$, $R^2=\Lambda^2V$. Then the quiver of $A(\J)$ is ${\bf Q}_3'$ and 
$\beta_3\alpha_3$, $\beta_4\alpha_4$, $\beta_4\alpha_3$, $\beta_3\alpha_4$ and $\beta_1\alpha_3$ are linearly independent elements in $rad^2 A(\J)/rad^3 A(J)$.

(b) Let $\g_{ss}=\so_8\oplus \so_{10}$, $R=\Gamma^+_1\boxtimes \Gamma^+_2$, $R^2=\Lambda^3W$. Then the quiver of $A(\J)$ is ${\bf Q}_4'$ and 
$\gamma_1\beta_1$, $\gamma_1\beta_2$, $\gamma_3\beta_1$, $\gamma_3\beta_2$, $\gamma_1\beta_4$ and $\gamma_3\beta_4$ are linearly independent elements in 
$rad^2 A(\J)/rad^3 A(J)$.
\end{Cor}

\begin{proof} It follows from computation of $P^2(L)$ for certain simple $L$. For example, to prove (a) take $L$ to be a non-trivial simple submodule in 
$\Gamma^+_2\otimes \Gamma_2^+=\Lambda^2 W\oplus \Lambda^+W\oplus tr$. Lemma \ref{mixedso8} implies $\Lambda^2 W\oplus \Lambda^+W\oplus tr\subset P^2(L)$.
On the other hand, it is shown in Lemma \ref {quivers-mixed} (1) that $\Lambda^2V\subset P^2(\Lambda^2W)$. Now (a) follows.
The proof of (b) is left to the reader.
\end{proof}

\begin{thm}
If $[\rr,\rr]\neq 0$, then $A(J)$ is wild.
\end{thm}
\begin{proof}
One has to show that if $\g$ is a Lie algebra from (1)-(4)  {\bf List B}, then $A(\J)$ is wild. 
For any simple $L\in\gm$ denote by $e_{L}$
the idempotent corresponding to the projector onto $P(L)$. In all cases we use the same method. 
We consider $B=eA(\J)e$ for some idempotent $e\in A(\J)$ and show  
that $B$ is wild. Then by Lemma \ref{lemma-pro-subquiver} $A(\J)$ is wild.

Let $\g=\so_{2m+1}\niplus(V\niplus\Lambda^2 V)$. Put
$$B=(e_{\Lambda^m V}+e_{\Lambda^{m-1} V})A(\J)(e_{\Lambda^m V}+e_{\Lambda^{m-1} V}).$$
By Lemma\,\ref{jstandard2} $B={\bf k}Q/I$, where
\begin{equation}
Q:\quad \xymatrix{\Lambda^{m-1} V\ar@/^0.4pc/[r]^{\gamma_{m-1}}\ar@{<-}^{\gamma_{m-2}\delta_{m-2}}@(ul,ur) & 
\Lambda^m V\ar@{<-}^{\gamma_m}@(ul,ur)\ar@/^0.4pc/[l]^{\delta_{m-1}}}
\end{equation} 
and $I\subset \gamma_{m-2}\delta_{m-2} rad B+rad^3 B$. Then $B$ has a factor algebra isomorphic to $A_1$, see \eqref{A_1_A_2}, 
and by Lemma\,\ref{han} is wild.

For $\g=\so_{2m}\niplus(V\niplus\Lambda^2 V)$ set 
$$B=e_{\Lambda^{m-1} V}A(\J)e_{\Lambda^{m-1} V}.$$
Then Lemma \ref{jstandard2} implies that $\gamma^+\delta^+$, $\gamma^-\delta^-$, $\delta_{m-2}\gamma_{m-2}$ are linearly independent 
elements in $rad B/rad^2 B$. Thus the quiver of $B$ is one vertex with at least three loops. 
From Theorem\,\ref{Gabriel}\eqref{double-quiver} 
it follows that $B$ is wild.

Let $\g_{ss}=\so_8$, $R=\Lambda^+ V$ and $R^2=\Lambda^2V$. We set
$$B=(e_{\Lambda^+V}+e_{\Lambda^2V})A(\J)(e_{\Lambda^+V}+e_{\Lambda^2V}).$$
Lemma\,\ref{so8unital2} implies the quiver of $B$ is \eqref{quiver-lambda-plus1}. Furthermore, $B$ has a quotient isomorphic to $A_2$, see \eqref{A_1_A_2},
by Lemma\,\ref{han} $B$ is wild.

Let $\g_{ss}=\so_8\oplus\so_8$, $R=\Gamma_1^+\boxtimes\Gamma_2^+$, $R^2=\Lambda^2V$. Set 
$$e_1=e_{\Lambda^2 V}+e_{\Lambda^2 W}+e_{\Lambda^+ W},$$
$B=e_1A(\J)e_1$ and apply Corollary \ref{mixed2} (a). Then $B$ has the quiver
\begin{equation}
\xymatrix{\Lambda^2 V\ar[r]^{\beta_3\alpha_1} & \Lambda^2 W\ar@/^0.4pc/[r]^{\beta_4\alpha_3}\ar@{<-}^{\beta_3\alpha_3}@(ul,ur) & 
\Lambda^+ W\ar@{<-}^{\beta_4\alpha_4}@(ul,ur)\ar@/^0.4pc/[l]^{\beta_3\alpha_4}}
\end{equation} 
By Theorem\,\ref{Gabriel}\eqref{double-quiver} $B/rad^2B$ is wild, hence $B$ is wild.

Finally, we consider $\g_{ss}=\so_{8}+\so_{10}$, $R=\Gamma_2^+\boxtimes\Gamma_2^+$, $R^2=\Lambda^3W$.
Let $$e_2=e_{W}+e_{\Lambda^2 W}+e_{\Lambda^3 W}+e_{\Lambda^+ W}+e_{\Lambda^+ V}+e_{\Gamma^+_1\boxtimes\Gamma_2^-},$$ $B=e_2 A(\J)e_2$. 
Corollary \ref{mixed2}(b) implies that $B$ is the path algebra of the quiver 
$$
\xymatrix{
W& \Lambda^2 W \ar[d]_{\beta_1} & \Lambda^4 W \ar[dl]_{\beta_2}\\
\Lambda^- W & {\Gamma_1^+\boxtimes\Gamma_2^-}\ar[l]^{\gamma_3}\ar[ul]_{\gamma_1} &  \Lambda^+ V \ar[l]^{\beta_4}}
$$
By Theorem\,\ref{Gabriel}\eqref{Dynkin} $B$ is wild.

\end{proof}

\begin{Cor}\label{only-zero-rad}
Let $\g$ be such that $Jor(\g)$-mod$_1$ is tame then $\hat{\g}=\g$.
\end{Cor}

\section{Classification theorem: general case}

\begin{thm}\label{thm-general-case}
Let $\J$ be a unital Jordan algebra such that $\J_{ss}$ is a direct sum of Jordan algebras of bilinear 
form over vector space of dimension greater than $4$. Then
$\J$ is tame if and only if $Lie(\J)$ is a direct sum of Lie algebras from  {\bf List A} and simple orthogonal algebras. 
\end{thm}

Let $\displaystyle\g=\bigoplus_{i=1}^r \g(i)$ be a direct sum of Lie algebras with short grading. Then the category $\gm$ has a 
decomposition in the direct sum 
$$\bigoplus_{i=1}^r\g(i)-\text{mod}_1\oplus \bigoplus_{i<j\leq r}\mathcal S_{i,j},$$
where $\mathcal S_{i,j}$ is the category of $\g(i)\oplus\g(j)$-modules which have very short grading over $\g(i)$ and $\g(j)$.
Simple objects in $\mathcal S_{i,j}$ are isomorphic to $L_1\boxtimes L_2$, where $L_1$ is a simple object 
$\g(i)$-mod$_{\frac12}$ and $L_2$ is a simple object 
$\g(j)$-mod$_{\frac12}$. 

\begin{lem}\label{projective-tesor} If $P(L)$ and $P(L')$ are projective covers of $L$ and $L'$ in $\g(i)$-mod$_{\frac12}$ and 
$\g(j)$-mod$_{\frac12}$ respectively, then $P(L)\boxtimes P(L')$ is the projective cover $L\boxtimes L'$
in $\mathcal S_{i,j}$.
\end{lem}

\begin{proof} As it was explained in Section \ref{subsection-plus} 
$$P(L\boxtimes L')=(I(L\boxtimes L'))^{sh}.$$
Since $U(\g(i)\oplus\g(j))=U(\g(i))\otimes U(\g(j))$ we have
$$I(L\boxtimes L')=I(L)\boxtimes I(L').$$
Furthermore,
$$(I(L)\boxtimes I(L'))^{sh}=(I(L))^{sh}\boxtimes (I(L'))^{sh}=P(L)\boxtimes P(L').$$
\end{proof}

\begin{Cor}
Let $P$ and $P'$ are projective generators in $\g(i)$-mod$_{\frac12}$ and $\g(j)$-mod$_{\frac12}$ respectively.
Then $P\boxtimes P'$ is a projective generator in $\mathcal S_{i,j}$ and 
$$\operatorname{End}_{\g}(P\boxtimes P')\simeq  
\operatorname{End}_{\g(i)}(P)\otimes \operatorname{End}_{\g(j)}(P').$$
\end{Cor}

\noindent Now we can prove Theorem\,\ref{thm-general-case}
\begin{proof}
It is sufficient to show that if $\g(i)$ and $\g(j)$ are two algebras from {\bf List A} then $\mathcal S_{i,j}$ is tame
(if $\g(i)$ is simple, this is trivial).
First we construct the projective indecomposable modules in $\g$-mod$_{\frac12}$, where $\g$ is one of the algebra from {\bf List A}.
By Lemma\,\ref{quiver_Lie} the Ext quiver $Q^{\frac12}(\g)$ is the following:
$$\begin{array}{ccccccc} Q^{\frac12}(\so_{2m+1}\niplus V) &&  Q^{\frac12}(\so_{2m}\niplus V) &&   Q^{\frac12}(\so_8\niplus \Lambda^{+} V) 
 && Q^{\frac12}((\so_8\oplus\so_8)\niplus\Gamma_1^{+}\boxtimes\Gamma_2^{+})\\ \xymatrix{\Gamma\ar_{\ }@(ur,ul)} && \xymatrix{\Gamma^+\ar@/^0.4pc/[r]^{\ } & 
\Gamma^-\ar@/^0.4pc/[l]} && \xymatrix{\ ^{\ ^{\ }}\Gamma^+ \ar_{\ }@(ur,ul)& 
\Gamma^-}&&  \xymatrix{\Gamma_1^+\ar@/^0.4pc/[r]^{\ } & 
\Gamma_2^+\ar@/^0.4pc/[l] &\Gamma_1^-\quad\Gamma_2^-} \end{array}$$
$$\begin{array}{ccccc} Q^{\frac12}((\so_8\oplus\so_{10})\niplus\Gamma_1^{+}\boxtimes\Gamma_2^{+}) & \xymatrix{\Gamma_1^+\ar[r] & 
\Gamma_2^+\\ \Gamma_1^- & 
\Gamma_2^-\ar[ul] } &&
Q^{\frac12}((\so_{10}\oplus\so_{10})\niplus\Gamma_1^{+}\boxtimes\Gamma_2^{+})&\xymatrix{\Gamma_1^+ & 
\Gamma_2^+\\ \Gamma_1^-\ar[ur] & 
\Gamma_2^-\ar[ul] }\end{array}$$

\medskip

\noindent Observe that for any algebra $\g$ from {\bf List A} Lemma \ref{special} implies that any projective indecomposable module in 
$\g$-mod$_{\frac12}$ has radical 
filtration of the length at most two. Therefore 
it is completely determined by $Q^{\frac12}(\g)$. Moreover for any simple module $L\in\g$-mod$_{\frac12}$ 
its projective cover $P(L)$ is either $L$ or $P(L)=\frac{L}{K}$, where $P^1(L)=K$ is simple $\g$-mod$_{\frac12}$ module.

Let $P=P(L)$ and $P'=P(L')$ be projective indecomposable modules in 
$\g(i)$-mod$_{\frac12}$ and $\g(j)$-mod$_{\frac12}$ respectively. Then 
$P\boxtimes P'$ is one of the following 
$$
\begin{array}{ccccc}
L\boxtimes L' && \begin{array}{c} L\boxtimes L'\\ \overline{\ \ P^1(L)\boxtimes L'\ \ } \end{array}   && 
\begin{array}{c} L\boxtimes L'\\ \overline{\ P^1(L)\boxtimes L'+ L\boxtimes \ P^1(L')}\\ 
\overline{\quad \ P^1(L)\boxtimes \ P^1(L')\quad}  \end{array} 
\end{array}
$$
Since $P^1(L), P^1(L')$ are simple, $P^1(L)\boxtimes P^1(L')\in {\mathcal S}_{i,j}$ is also simple. Hence the associative algebra of the category
$\mathcal S_{i,j}$ satisfies the conditions of 
Theorem\,\ref{Viktor}. One can check that if $\g(i)$, $\g(j)$ are from {\bf List A}, then the double quiver of the Ext quiver of $\mathcal S_{i,j}$ 
is a disjoint union of Dynkin and extended Dynkin diagrams, therefore $\mathcal S_{i,j}$ is either tame or finite.
This finishes the proof.
\end{proof}
 
\begin{ex}
The Ext quiver for the category $\mathcal S_{i,j}$ if $\g(i)=\so_{2m+1}\niplus V$ and $\g(j)=\so_{2m}\niplus V$ is
$$\xymatrix 
{{\Gamma_i\boxtimes\Gamma^+_j}\ar@(ul,ur)^{\alpha}\ar@/^0.4pc/[r]^{\delta} & 
 {\Gamma_i\boxtimes\Gamma^-_j}\ar@(ul,ur)^{\beta}\ar@/^0.4pc/[l]^{\nu}} \qquad 
\begin{array}{c} \delta\alpha=\beta\nu  \\ {\rm all\ other\ path\ of\ length\ two\ are\ zero}\end{array}
$$

\noindent The Ext quiver for the category $\mathcal S_{i,j}$ if $\g(i)=(\so_8\oplus\so_{10})\niplus\Gamma_1^+\boxtimes\Gamma_2^+$ and 
$\g(j)=\so_{2m}\niplus V$ is
$$\xymatrix 
{{\Gamma_2^-\boxtimes\Gamma^+_j}\ar[r]^{\alpha_1}\ar@/^0.4pc/[d]^{\delta_1} & 
{\Gamma_1^+\boxtimes\Gamma^+_j}\ar[r]^{\alpha_2}\ar@/^0.4pc/[d]^{\delta_2} 
& {\Gamma_2^+\boxtimes\Gamma^+_j}\ar@/^0.4pc/[d]^{\delta_3} & {\Gamma_1^-\boxtimes\Gamma^+_j}\\
{\Gamma_2^-\boxtimes\Gamma^-_j}\ar[r]^{\beta_1}\ar@/^0.4pc/[u]^{\nu_1} & 
{\Gamma_1^+\boxtimes\Gamma^-_j}\ar[r]^{\beta_2}\ar@/^0.4pc/[u]^{\nu_2}
&{\Gamma_2^+\boxtimes\Gamma^-_j}\ar@/^0.4pc/[u]^{\nu_3} & {\Gamma_1^-\boxtimes\Gamma^-_j}\\} \qquad 
\begin{array}{c} \\ \beta_i\delta_i=\delta_{i+1}\alpha_i\\
\beta_i\nu_i=\nu_{i+1}\alpha_i, =1,2; \\ {\rm all\ other\ path\ of\ length}\\
{\rm two\ are\ zero}\end{array}
$$
\end{ex}

\section{Appendix: Representations of quivers}

In this section we will collect some notions, theorems and methods which will be used to determine the representation type of algebras
given as a quiver with relations.

Let $\mathcal C$ be an abelian category and $P$ be a projective generator in $\mathcal C$.
It is well-known fact (see \cite{Gelf-Manin} ex.2 section 2.6) that the functor $\operatorname{Hom}_{\mathcal C}(P,M)$
provides an equivalence of $\mathcal C$ and the category of right modules over the ring
$A=\operatorname{Hom}_{\mathcal C}(P,P)$. In case when every object in $\mathcal C$  has the finite length and each simple object 
has a projective cover, one reduces the problem of classifying indecomposable objects in $\mathcal C$ 
to the similar problem for modules over a finite-dimensional algebra (see \cite{Gab,Gab2}).
If $L_1,\dots,L_r$ is
the set of all up to isomorphism simple objects in $\mathcal C$
and $P_1,\dots,P_r$ are their projective covers, then $A$ is a
pointed algebra which is usually realized as the path algebra of a
certain quiver $Q$ with relations. The vertices of $Q$ correspond
to simple (resp. projective) modules and the number of arrows from
vertex $i$ to vertex $j$ equals
$\operatorname{dim}\operatorname{Ext}^1(L_j,L_i)$ (resp.
$\operatorname{dim}\operatorname{Hom}(P_i,\operatorname{rad}P_j/\operatorname{rad}^2 P_j$)).

We apply this approach to the case when $\mathcal C$ is  $\gm$.

For any quiver $Q$ let $V(Q)$ denote the 
set of vertices of $Q$ and $Ar(Q)$ the set of its arrows.
The {\it quiver double} $D(Q)$
of the quiver $Q$ is defined
as follows:
\begin{equation}\label{equation-def-double-quiver}
V(D(Q))=\{\,X^{+}, X^{-}|\,X\in V(Q)\}, \
Ar(D(Q))=\{\,\tilde{a}: X^{-} \mapsto Y^{+}\, |\, \text {if\ }a:\,X\mapsto Y\in Ar(Q)\}.
\end{equation}

The following results are
classical.

\begin{thm}\label{Gabriel}
\begin{enumerate}
\item\label{Dynkin} Let $A=\kk[Q]$ is
the path algebra of a quiver $Q$. Then $A$ if of finite (tame)
representation type if and only if $Q$ is a disjoint union of
oriented Dynkin diagrams (extended Dynkin diagrams).

\item\label{double-quiver}
Let $A$ be a finite dimensional
associative algebra, such that
$rad^{2}A=0$, $Q$ its quiver.
Then $A$  is of finite (tame)
representation type if and only
if $D(Q)$ is a disjoint union
of oriented Dynkin diagrams
(extended Dynkin diagrams).
\end{enumerate}
\end{thm}

\begin{lem}\label{lemma-pro-subquiver}\cite[1.4.7]{Erdmann} Suppose $A$ is a finite-dimensional algebra.
\begin{enumerate}
\item If $e$ is an idempotent of $A$ such that $eAe$ is wild then so is $A$.
\item  Let $I$ be an ideal of $A$. If $A/I$ is wild then $A$ is wild as well.
\end{enumerate}
\end{lem}

\begin{lem}\label{Viktor1} Let $A={\bf k}Q/I$, $e$ be an indecomposable idempotent and 
$P=Ae$ is both projective and injective. Assume that $rad^3 P=0$, while $rad^2 P\neq 0$.
Then any indecomposable $A$-module $M$ such that $rad^2 eM\neq 0$ is isomorphic to $P$.
\end{lem}
\begin{proof}
Injectivity of $P$ implies that $rad^2 P$ is simple and coincides with the socle of $P$. 
Let $v\in M$ be such that $rad^2A ev\neq 0$. Then $rad^2 Aev=rad^2 P$ and therefore
$Aev\simeq P$. Since $P$ is injective, we obtain that $M=Aev\simeq P$.
\end{proof}
 
Next we determine representation type of algebras whose indecomposable projective modules satisfy 
the condition of the Lemma\,\ref{Viktor1}.
\begin{thm}\label{Viktor}
Suppose that $A={\bf k}Q/I$ and  any indecomposable projective module $P$ such that $rad^2 P\neq 0$ satisfies the conditions of 
Lemma\,\ref{Viktor1}. Then $A$ is of finite (respectively tame) representation type if and only if the double quiver $D(Q)$ is a 
disjoint union of Dynkin diagrams (respectively extended Dynkin diagrams).
\end{thm}
\begin{proof} 
Let $M$ be an indecomposable $A$-module and $rad^2 M\neq 0$
then by Lemma\,\ref{Viktor1} $M$ is projective. Otherwise $M$ is 
a module over $A/rad^2A$. The statement follows 
from Theorem\,\ref{Gabriel} \eqref{double-quiver}.
\end{proof}

\begin{lem}\label{divide}  Assume that $V(Q)$ is a disjoint 
union of $S_l$, $S_r$ and $T$.
Assume $Q(T)$ is disjoint, any path from $S_l$ to $S_r$ (or from $S_r$ to $S_l$) contains a vertex from $T$ 
and any path from $S_r$ to $S_r$ and from $S_l$ to $S_l$ does not contain a vertex from $T$. 
Let $A={\bf k}Q/I$ and any path from $S_l$ to $S_r$ (or from $S_r$ to $S_l$) belongs to $I$. 
Then for any indecomposable $A$-module $M$ either $e_{S_l} M=0$ or $e_{S_r} M=0$, where $e_{S_r}=\sum_{i\in S_r} e_i$ and 
$e_{S_l}=\sum_{i\in S_l} e_i$.
\end{lem}
\begin{proof}
Let $e_{T}=\sum_{i\in T} e_i$. One can check that both $e_{S_r}M+e_{T}Ae_{S_r}M$ and $e_{S_l}M+e_{T}Ae_{S_l}M$ split as direct summands. 
Hence one of them is zero.
\end{proof}
 
Recall that for any associative algebra $A={\bf k}Q/I$ of finite global dimension 
the {\it Tits form of} $A$ is the quadratic form
$q_A:\Z^{V(Q)}\to\Z$ which is defined by
\begin{equation}\label{Tits-form}
q_A(x)=\sum_{i\in V(Q)} x_i^2-\sum_{i\to j\in E(Q)} x_ix_j +\sum_{i,j\in V(Q)} g(i,j) x_i x_j,
\end{equation}
where $g(i,j)=|G\cap e_j Ie_i|$ for a minimal set $G\subset \bigcup_{i,j\in V(Q)} e_j Ie_i$ of generators of the ideal $I$.

A quiver $Q$ is called a {\it tree} if its underlying graph is a tree (i.e. does not contain cycles), the algebra 
$A={\bf k}Q/I$ is a {\it tree algebra} if $Q$ is a tree.

\begin{thm}\label{Bruestle}\cite[Thm 1.1]{Bru}
Let $A$ be a tree algebra. Then $A$ is tame precisely when the Tits form $q_A$ is weakly non-negative.
\end{thm}

In \cite{Han} Y.~Han has classified tame two-point algebras and minimal wild two-point algebras.
We list
the following two algebras from Table W, \cite{Han}, $A_1={\bf k}Q/I_1$ and $A_2={\bf k}Q/I_2$, where
\begin{equation}\label{A_1_A_2}
\xymatrix 
{Q:\ &\bullet\ar@/^0.4pc/[r]^{\mu}\ar@{<-}^{\alpha}@(dl,ul) & \bullet\ar@{<-}_{\beta}@(dr,ur)\ar@/^0.4pc/[l]^{\nu}} \quad 
\begin{array}{l}
I_1=\langle \alpha^2=\mu\alpha=\alpha\nu=\nu\beta=\mu\nu=\beta^2=0\rangle,\\
I_2=\langle \mu\alpha-\beta\mu=\alpha^2=\nu\mu=\alpha\nu=\nu\beta=\beta^3=\beta^2\mu=0\rangle.
\end{array}
\end{equation}
\begin{lem}\label{han}\cite[Thm 1]{Han}
Let $B={\bf k}Q/I$ be a two-point algebra which has either algebra $A_1$ or algebra $A_2$ as a factor, then $B$ is wild.
\end{lem}

\section{Acknowledgment} The first author was supported by the CAPES 2013/07844-0.
The second author was partially supported by NSF grant 1303301.
Both authors express their gratitude to Viktor Bekkert, our quiver expert.

\end{document}